\documentclass[12pt]{article}

\usepackage{epsfig,epsfig}
\usepackage{epstopdf}
\usepackage{amsmath}
\usepackage{mathrsfs}
\usepackage{amsfonts}
\usepackage{bm}
\usepackage{amssymb}
\usepackage{amsthm}
\usepackage{stmaryrd}
\usepackage{booktabs}
\usepackage{color}
\usepackage{multirow}
\usepackage{graphicx}
\usepackage{subfigure}
\usepackage{float}
\usepackage{appendix}
\usepackage{tikz}
\usetikzlibrary{calc}
\usetikzlibrary{matrix}
\usepackage{mathtools}
\usepackage{algorithm,algpseudocode,float}
\usepackage{lipsum}


\usepackage{lineno}
\usepackage{srcltx,graphicx}

\newtheorem{thm}{Theorem}[section]
\newtheorem{prop}[thm]{Proposition}

\newtheorem{example}[thm]{Example}

\usepackage{hyperref}

\numberwithin{equation}{section} \topmargin=-2cm \oddsidemargin=1cm
\begin{document}
\baselineskip=1.5pc
\title{\textbf{Well-balanced fifth-order finite difference Hermite WENO scheme for the shallow water equations}}
\author{
Zhuang Zhao\footnote{School of Mathematical Sciences and Institute of Natural Sciences, Shanghai Jiao Tong University, Shanghai, 200240, China. E-mail: zzhao-m@sjtu.edu.cn. The research of this author is partially supported by the Postdoctoral Science Foundation of China (Grant Nos. 2021M702145). }
~and~
Min Zhang\footnote{School of Mathematical Sciences, Peking University, Beijing, 100871, China. E-mail: minzhang@math.pku.edu.cn. The research of this author is partially supported by the Postdoctoral Science Foundation of China (Grant Nos. 2022M710229).}}

\date{}

\maketitle

\begin{abstract}
In this paper, we propose a well-balanced fifth-order finite difference Hermite WENO (HWENO) scheme for the shallow water equations with non-flat bottom topography in pre-balanced form.
For achieving the well-balance property, we adopt the similar idea of WENO-XS scheme [Xing and Shu, J. Comput. Phys., 208 (2005), 206-227.] to balance the flux gradients and the source terms.
The fluxes in the original equation are reconstructed by the nonlinear HWENO reconstructions while other fluxes in the derivative equations are approximated by the high-degree polynomials directly.
And an HWENO limiter is applied for the derivatives of equilibrium variables in  time discretization step to control spurious oscillations which maintains the well-balance property. Instead of using a five-point stencil in the same fifth-order WENO-XS scheme, the proposed HWENO scheme only needs a compact three-point stencil in the reconstruction.
Various benchmark examples in one and two dimensions are presented to show the HWENO scheme is fifth-order accuracy, preserves steady-state solution, has better resolution, is more accurate and efficient, and is essentially non-oscillatory.
\end{abstract}


\vspace{5pt}


\normalsize \vskip 0.2in
\newpage

\section{Introduction}
In this paper, we are interested in designing a high-order finite difference Hermite weighted essentially non-oscillatory (HWENO) scheme for shallow water equations with non-flat bottom topography.
The shallow water equations (SWEs), also referred to as the Saint-Venant system, model
the water flow over a surface. It plays an important role
in the ocean and hydraulic engineering, such as hydraulic jumps/shocks, open-channel flows,
bore wave propagation, tidal flows in the estuary and coastal zones.
The SWEs in conservative form are read as
\begin{equation}\label{swe-form}
\frac{\partial}{\partial t}
\begin{bmatrix*}[l]
  h\\
  hu\\
  hv\\
\end{bmatrix*}
+\frac{\partial}{\partial x}
\begin{bmatrix*}[c]
hu\\
 hu^2+\frac{1}{2}gh^2\\
huv\\
\end{bmatrix*}
+\frac{\partial}{\partial y}
\begin{bmatrix*}[c]
 hv\\
 huv\\
hv^2+\frac{1}{2}gh^2\\
\end{bmatrix*}
=
\begin{bmatrix*}[c]
  0\\
  -ghb_x\\
  - ghb_y\\
\end{bmatrix*},
\end{equation}
where $h(x,y,t)\geq0$ is the depth of water, $(hu, hv)$ are the discharges,
$(u,v)$ are the velocities,
$b = b(x,y)$ is the bottom topography assumed to be a given time-independent function, and
$g$ is the gravitation acceleration.

The homogeneous SWEs are equivalent to that of the isentropic Euler equations. However, the properties of the SWEs change a lot due to the presence of the source term. A distinct feature of the SWEs is that they admit steady-state solutions when the flux gradients are balanced by the source term exactly.
Of particular interest is the steady-states corresponding to still water (also called ``lake-at-rest"),
\begin{equation}\label{swe-case1}
 u=0,\quad v=0,\quad h+b = Const.
\end{equation}
It is crucial that this solution is preserved by numerical methods for the SWEs.
Indeed, many physical phenomena, such as waves on a lake or tsunami waves in the deep ocean, can be described as small perturbations of this lake-at-rest steady-state.
Traditional numerical schemes with a straight-forward handling of the source term cannot balance the effect of the source term and the flux. They are difficult to capture the steady-state well numerically unless the method can preserve the steady-state solution. This property is known as the C-property or well-balance property.

The concept of the ``exact C-property" was first introduced by Bermudez and Vazquez \cite{Bermudez-Vazquez-1994} in 1994. Since then, many well-balanced schemes have been developed for the SWEs.
For example, LeVeque \cite{LeVeque-1998JCP} developed a quasi-steady wave propagation algorithm which introduced a Riemann problem in the center of each grid cell such that the flux difference exactly cancels the source term.
Zhou et al. \cite{Zhou-etal-2001JCP} proposed a surface gradient method for the treatment of the source terms based on an accurate reconstruction of the conservative variables at cell interfaces.
Audusse et al. \cite{Audusse-etal-2004Siam} designed a second-order well-balanced scheme in terms of a hydrostatic reconstruction idea.
Tang et al. \cite{TangTX-2004} extended a Kinetic Flux Vector Splitting (KFVS) scheme to solve the SWEs with source terms.
Vukovic and Sopta \cite{Vukovic-Sopta-2002JCP} proposed the finite difference ENO and WENO schemes with the source term decomposed, where the ENO and WENO reconstruction are applied to both the flux and the source term.
Xing and Shu designed the fifth-order well-balanced finite difference \cite{Xing-Shu-2005JCP} based on a special decomposition of the source term, then, they also designed the well-balanced finite volume WENO scheme and DG method \cite{Xing-Shu-2006CiCP} for a class of hyperbolic balance laws including SWEs based on the hydrostatic reconstruction \cite{Audusse-etal-2004Siam}.
Caleffi \cite{Caleffi-2011} developed a well-balanced fourth-order finite volume Hermite WENO scheme for the one-dimensional SWEs on the basis of \cite{QS-HWENO-2004}.
For more related well-balanced high-order methods, e.g., finite difference schemes \cite{ChengLCS-2022,GaoH-2017-FD,HuangXX-2022,Li-DonGao-2020,LuQiu-2011JSC,Wang-etal-2020},
finite volume schemes \cite{Capilla-Balaguer-2013,Li-Lu-Qiu-2012JSC,NoelleXS-2007,Xing-Shu-2006JCP},
and DG methods \cite{Li-etal-2018JCAM,Xing-Zhang-Shu-2010,Xing-Zhang-2013JSC,Zhang-Huang-Qiu-2021JSC,Zhang-Huang-Qiu-2020CiCP,Zhang-Xia-Xu-2021JSC}.

The HWENO schemes were first introduced by Qiu and Shu \cite{QS-HWENO-2004,QS-HWENO-2005} for solving hyperbolic conservation laws, which evolve both the function value and derivative of the governing variables in time, rather than only evolving the conservative variables in the WENO scheme \cite{js}. Since then, many HWENO schemes have been developed for hyperbolic conservation laws \cite{Capdeville,LiMRHW1,LiuQ1,MW,TLQ,ZA,weHTENO}.
The main advantage of the HWENO scheme is the compactness in the spatial reconstruction, which allows easier treatment for the boundary conditions and internal interfaces.
And the numerical results also show that the HWENO schemes are more accurate and have better resolution near discontinuities or internal interfaces than the same order traditional WENO schemes.
Motivated by these good properties, we devote to designing a well-balanced HWENO scheme to solve the SWEs.
One work \cite{Caleffi-2011} has been done in using HWENO scheme to solve the SWEs with non-flat bottom topography, in which Caleffi \cite{Caleffi-2011} developed a well-balanced finite volume HWENO method in one-dimensional case, but it only has the fourth-order accuracy and loses the fifth-order accuracy of the original HWENO scheme \cite{QS-HWENO-2004}.
Drawback of the finite volume HWENO scheme \cite{QS-HWENO-2004} is it cannot be extended to two dimensions straightforwardly by dimension-by-dimension manner.

In this paper, we generalize the finite difference HWENO scheme \cite{ZhaoZhuang-2020JSC-FD} to obtain a fifth-order well-balanced HWENO scheme for the one- and two-dimensional SWEs.
For achieving the well-balance property, we use the similar idea of the fifth-order finite difference WENO (WENO-XS) scheme \cite{Xing-Shu-2005JCP} to balance the flux gradients and the source terms in the spatial discretization step. The fluxes in the original equation are reconstructed by nonlinear HWENO reconstructions while other fluxes in the derivative equations are approximated by high-degree polynomials directly. To control spurious oscillations, an HWENO limiter is used to modified the derivative of the equilibrium variables in time discretization step, and the limiter procedure preserves the steady-state solution.

It is worth pointing out that although the HWENO scheme needs to solve the derivative equations which adds extra computational costs into the algorithm, the HWENO scheme is more efficient than the WENO-XS scheme in the sense that the former leads to a smaller error than the latter for a fixed amount of the CPU time (cf. Example \ref{test5-1d} and Example \ref{test1-2d} in \S\ref{sec:numerical}).
In addition, the proposed HWENO scheme is more compact than the WENO-XS scheme \cite{Xing-Shu-2005JCP} in the reconstruction.
To be specific, the HWENO scheme only needs a compact three-point stencil while the WENO-XS scheme needs a five-point stencil in the reconstruction for achieving fifth-order accuracy.
The numerical results in one and two dimensions show that the proposed HWENO scheme is efficient, has better resolution, keeps non-oscillatory, preserves steady-state solution, and has fifth-order accuracy in smooth regions.

The remainder of the paper is organized as follows.
A well-balanced fifth-order finite difference HWENO scheme for the one-dimensional SWEs is developed in \S\ref{sec:WBFDHWENO-1d}, and the extension to two dimensions by dimension-by-dimension manner is described in \S\ref{sec:WBFDHWENO-2d}.
One- and two-dimensional numerical results are presented in \S\ref{sec:numerical} to show the well-balance property, accuracy, efficiency, resolution, and non-oscillation of the proposed HWENO scheme.
Finally, the conclusions and further comments are given in \S\ref{sec:conclusions}.

\section{Well-balanced HWENO scheme for 1D SWEs}
\label{sec:WBFDHWENO-1d}

In this section, we present a well-balanced fifth-order finite difference HWENO scheme for solving the one-dimensional SWEs with non-flat bottom topography.
Comparing with the WENO-XS scheme \cite{Xing-Shu-2005JCP}, the HWENO scheme not only uses the values of solution but also evolves its first-order derivative, which is more compact in the spatial discretization.
To be specific, the proposed fifth-order HWENO scheme only needs a compact three-point stencil while the fifth-order WENO-XS scheme needs a five-point stencil in the reconstruction. The compactness of the scheme not only allows easier treatment of the boundary conditions and the internal interfaces, but also has better resolution near discontinuities or internal interface with less transition points.

We consider the one-dimensional SWEs as
\begin{equation}\label{swe-form-1d}
\frac{\partial}{\partial t}
\begin{bmatrix*}[l]
  h\\
  hu\\
\end{bmatrix*}
+\frac{\partial}{\partial x}
\begin{bmatrix*}[c]
 hu\\
hu^2+\frac{1}{2}gh^2\\
\end{bmatrix*}
=
\begin{bmatrix*}[c]
  0\\
  -ghb_x\\
\end{bmatrix*},
\end{equation}
where $h(x,t)\geq0$ is the depth of water,
$u$ is the velocity,
$b = b(x)$ is the bottom topography assumed to be a given time-independent function, and
$g$ is the gravitation acceleration. The still water steady-state solution of $\eqref{swe-form-1d}$ is
\begin{equation}\label{st-1d}
h+b = Const,\quad m=hu=0.
\end{equation}

Let $\bm{U} = (\eta = h+b,m=hu)^{T}$ and $\tilde{\bm{b}} = (0,b)^{T}$, and rewrite \eqref{swe-form-1d} into pre-balanced form as
\begin{equation}\label{swe-form-1d-e}
\begin{split}
&\frac{\partial \bm{U}}{\partial t}
+\frac{\partial}{\partial x}
\mathbf{F}(\bm{U},b)
= \bm{S}(\bm{U},b),\\
&\mathbf{F}(\bm{U},b) =\begin{bmatrix*}[c]
 m\\
  \frac{m^2}{\eta-b}+\frac{1}{2}g\big(\eta^2-2\eta b)\\
\end{bmatrix*}
,\quad \bm{S}(\bm{U},b)=
\begin{bmatrix*}[c]
  0\\
  -g\eta b_x\\
\end{bmatrix*} = -g\eta \tilde{\bm{b}}_x.
\end{split}
\end{equation}
%

To construct the HWENO scheme, we take partial
derivative w.r.t. $x$ on both sides of \eqref{swe-form-1d-e} and denote $ \bm{V} = (H,M)^{T}=(\eta_x,m_x)^{T}$ and $\widetilde{\bm{B}} = (0,B)^{T}= (0,b_x)^{T}=\tilde{\bm{b}}_x$. Then we have
\begin{subequations}\label{swe-form-1d-h}
\begin{align}
&\frac{\partial \bm{U}}{\partial t}
+\frac{\partial}{\partial x}\mathbf{F}(\bm{U},b)
= -g\eta \tilde{\bm{b}}_x, \label{swe-form-1d-hU}
\\
&\frac{\partial \bm{V}}{\partial t}
+\frac{\partial}{\partial x}\mathbf{\Phi}(\bm{U},b;\bm{V},B)
=-gH\tilde{\bm{b}}_x-g\eta\widetilde{\bm{B}}_x, \label{swe-form-1d-hV}
\end{align}
\end{subequations}
where
\begin{equation*}\label{swe-form-1d-h2}
\begin{split}
&\mathbf{\Phi}=\begin{bmatrix*}[c]
  M\\
   \big(g(\eta-b)-\frac{m^2}{(\eta-b)^2}\big)H
   +\frac{2m}{\eta-b}M
   +\big(\frac{m^2}{(\eta-b)^2}-g\eta\big)B\\
\end{bmatrix*}=\frac{\partial \mathbf{F}}{\partial \bm{U}}\bm{V} + \frac{\partial\mathbf{F}}{\partial \tilde{\bm{b}}}\widetilde{\bm{B}}.
\end{split}
\end{equation*}

\subsection{A balance of the flux and source term}

Following the idea of \cite{Xing-Shu-2005JCP}, we would like to find a linear scheme for the 1D system \eqref{swe-form-1d-h}, where the linear scheme represents as that the spatial derivatives in \eqref{swe-form-1d-hU} and \eqref{swe-form-1d-hV} are approximated by the linear finite difference operators $\mathcal{D}_1$ and $\mathcal{D}_2$, respectively.

\begin{prop}
Linear scheme for the 1D system \eqref{swe-form-1d-h} satisfying the still water steady-state solution \eqref{st-1d} can maintain the well-balance property.
\end{prop}
\begin{proof}
For the still water steady-state solution \eqref{st-1d}, we have
\begin{equation}\label{s-1d}
\eta=h+b = Const,\quad m=hu=0, \quad H = \eta_x=0,\quad M=m_x=0.
\end{equation}
Then, the residues of \eqref{swe-form-1d-hU} and \eqref{swe-form-1d-hV} will be reduced as
\begin{equation*}\label{swe-form-1d-res-D}
\begin{split}
\mathbf{R} &=\begin{bmatrix*}[c]
  0\\
  -g\eta \mathcal{D}_1(b)
\end{bmatrix*} -
\begin{bmatrix*}[c]
0\\
\mathcal{D}_1\big(\frac{1}{2}g\big(\eta^2-2\eta b)\big)\\
\end{bmatrix*}
=
\begin{bmatrix*}[c]
0\\
\mathcal{D}_1\big(-\frac{1}{2}g\eta^2)\big)\\
\end{bmatrix*} = 0,
\\
\mathbf{R}_x &=\begin{bmatrix*}[c]
    0\\
-g\eta \mathcal{D}_2(B)\\
\end{bmatrix*} -
\begin{bmatrix*}[c]
   0\\
   \mathcal{D}_2\big(-g\eta B\big)
\end{bmatrix*}
=
\begin{bmatrix*}[c]
0\\
-\mathcal{D}_2\big(g\eta B-g\eta B\big)
\end{bmatrix*} = 0.
\end{split}
\end{equation*}
\end{proof}

\vspace{10pt}

Generally, the fifth-order finite difference HWENO scheme \cite{ZhaoZhuang-2020JSC-FD} is nonlinear.
The nonlinearity comes from the nonlinear weights, which in turn comes from the nonlinearity of the smooth indicators measuring the smoothness of the flux functions.
Following \cite{Xing-Shu-2005JCP}, we make minor modifications for the fifth-order finite difference HWENO scheme \cite{ZhaoZhuang-2020JSC-FD} to make them keep well-balance property, accuracy and nonlinear stability, simultaneously.

Now, we describe the procedure of the well-balanced fifth-order finite difference HWENO scheme, which is to balance the flux gradients and the source terms in the spatial discretization step, for the 1D SWEs.


\textbf{Step 1.} Discretize the flux gradients $\big(\frac{\partial\mathbf{F}}{\partial x},\frac{\partial\mathbf{\Phi}}{\partial x}\big)$ as
\begin{equation}\label{swe-form-1d-dx}
\begin{split}
&\frac{\partial\mathbf{F}}{\partial x}\approx \frac{\widehat{\mathbf{F}}_{i+\frac12}-\widehat{\mathbf{F}}_{i-\frac12}}{\Delta x},
\qquad
\frac{\partial\mathbf{\Phi}}{\partial x}\approx \frac{\widehat{\mathbf{\Phi}}_{i+\frac12}-\widehat{\mathbf{\Phi}}_{i-\frac12}}{\Delta x},
\end{split}
\end{equation}
where the numerical fluxes $\widehat{\mathbf{F}}_{i+\frac12}$ and $\widehat{\mathbf{\Phi}}_{i+\frac12}$ are the fifth-order approximation of flux functions $\mathbf{F}(\bm{U},b)$ and $\mathbf{\mathbf{\Phi}}(\bm{U},b;\bm{V},B)$, respectively, at the boundary $x_{i+\frac12}$ of the cell $I_i$.
For the stability requirement, we should split the fluxes $\mathbf{F}(\bm{U},b)$ and $\mathbf{\Phi}(\bm{U},b;\bm{V},B)$ into two parts for the upwinding mechanism as
\begin{equation}
\begin{split}
&\mathbf{F} = \mathbf{F}^{+}+\mathbf{F}^{-},\quad \mathbf{F}^{\pm} =
\frac{1}{2}\big(\mathbf{F}\pm \alpha \bm{U}\big),\\
&\mathbf{\Phi} = \mathbf{\Phi}^{+}+\mathbf{\Phi}^{-}, \quad
\mathbf{\Phi}^{\pm} = \frac{1}{2}\big(\mathbf{\Phi}\pm \alpha \bm{V}\big).
\end{split}
\end{equation}

In this work, the numerical fluxes $\widehat{\mathbf{F}}_{i+\frac12}^{\pm}$ are reconstructed by the nonlinear HWENO scheme using the local characteristic variables
while $\widehat{\mathbf{\Phi}}_{i+\frac12}^{\pm}$ are approximated by high-degree polynomials on each component directly.
We now describe the procedure of the fifth-order finite difference HWENO reconstruction \cite{ZhaoZhuang-2020JSC-FD} for $\widehat{\mathbf{F}}_{i+\frac12}^{\pm}$ and the linear approximation for $\widehat{\mathbf{\Phi}}_{i+\frac12}^{\pm}$.
Without loss generality, we here only describe the detailed HWENO reconstruction of $\hat f_{i+\frac12}$ and $\hat \phi_{i+\frac12}$ for $f'(u)\geq 0$ in the scalar case, where $\phi=f(u)_x$.
The procedure for the case of $f'(u)<0$ is the mirror symmetric with respect to $x_{i+\frac12}$.

Firstly, we choose a set of suitable candidate stencils and construct polynomials based on Hermite reconstruction on these stencils.
To be specific, we choose a big stencil $\mathcal{T} = \{x_{i-1},x_i,x_{i+1}\}$
and three small stencils $\mathcal{T}_1 = \{x_{i-1},x_i\}$,
$\mathcal{T}_2 = \{x_i,x_{i+1}\}$
and $\mathcal{T}_3 = \{x_{i-1},x_i,x_{i+1}\}$. 
Using the Hermite reconstruction on stencils $\mathcal{T}$, $\mathcal{T}_1$, $\mathcal{T}_2$  and $\mathcal{T}_3$, respectively, there are a unique fifth-degree polynomial $Q(x)$ and three cubic polynomials $q_1(x)$, $q_2(x)$, and $q_3(x)$, such that
\begin{equation}\label{HWEp1}
\begin{split}
&Q(x):~\begin{cases}
\frac{1}{\Delta x} \int_{I_{i+\ell}} q_0(x)dx = f_{i+\ell},\quad \ell=-1,0,1, \\ \frac{1}{\Delta x} \int_{I_{i+\ell}} q'_0(x)dx = \phi_{i+\ell},\quad\ell=-1,0,1,
\end{cases}
\\&q_1(x):~\begin{cases}
\frac{1}{\Delta x} \int_{I_{i+\ell}} q_1(x)dx = f_{i+\ell},\quad \ell=-1,0, \\
\frac{1}{\Delta x} \int_{I_{i+\ell}} q'_1(x)dx = \phi_{i+\ell},\quad\ell=-1,0,
\end{cases}
\\&q_2(x):~\begin{cases}
\frac{1}{\Delta x} \int_{I_{i+\ell}} q_2(x)dx = f_{i+\ell},\quad \ell=0,1, \\
\frac{1}{\Delta x} \int_{I_{i+\ell}} q'_2(x)dx = \phi_{i+\ell},\quad\ell=0,1,
\end{cases}
\\&q_3(x):~\begin{cases}
\frac{1}{\Delta x} \int_{I_{i+\ell}} q_3(x)dx = f_{i+\ell},\quad \ell=-1,0,1, \\
\frac{1}{\Delta x}\int_{I_{i}}~~ q'_3(x)dx = \phi_{i},
\end{cases}
\end{split}
\end{equation}
where $f_{i+\ell} = f(u_{i+\ell})$, $\phi_{i+\ell} = \phi(u_{i+\ell},v_{i+\ell})$, $\ell=-1,0,1$.
Evaluate the function values and the derivatives of these polynomials at the point $x_{i+\frac12}$, then, we have
\begin{equation*}
\begin{split}
q_1(x_{i+\frac12})&=\frac12\big(f_{i-1}+f_i\big)
+\frac{\Delta x}{6}\big(\phi_{i-1}+5\phi_i\big),\\
q_2(x_{i+\frac12})&=\frac12\big(f_{i}+f_{i+1}\big)
+\frac{\Delta x}{6}\big(\phi_{i}-\phi_{i+1}\big),\\
q_3(x_{i+\frac12})&=\frac{1}{12}\big(f_{i-1}+10f_{i}+f_{i+1}\big)
+\frac{\Delta x}{ 2}\phi_i,
\end{split}
\end{equation*}
and
\begin{equation*}
\begin{split}
Q'(x_{i+\frac12})&=\frac{1}{4\Delta x}\big(f_{i-1}-8f_i+7f_{i+1}\big)
+\frac{1}{12}\big(\phi_{i-1}-2\phi_i-5\phi_{i+1}\big).\\
\end{split}
\end{equation*}
Finally, the values of $\hat{f}_{i+\frac12}$ and $\hat h_{i+\frac12}$ are reconstructed by
\begin{align}
&\hat{f}_{i+\frac12} =\omega_1 q_1(x_{i+\frac12}) + \omega_{2} q_{2}(x_{i+\frac12})+ \omega_{3} q_{3}(x_{i+\frac12}),\label{Rec-f}
\\&\hat{\phi}_{i+\frac12} =Q'(x_{i+\frac12}). \label{Rec-phi}
\end{align}
Here, the nonlinear weights $\omega_\ell$ satisfy $\omega_\ell\geq 0$, $\omega_1+\omega_2+\omega_3=1$, defined as
\begin{equation}\label{GHYZZ}
\begin{split}
\omega_\ell=\frac{\tilde\omega_\ell}{\sum_{\ell=0}^{2}\tilde\omega_{\ell}},\quad\quad \tilde\omega_{\ell}=\frac{\gamma_{\ell}}{(\beta_{\ell}+\epsilon)^2},\quad \ell=1,2,3,
\end{split}
\end{equation}
where $\gamma_1=\frac{3}{10}$, $\gamma_2=\frac{3}{10}$, $\gamma_3=\frac{4}{10}$ are the linear weights, and $\epsilon$ is a small positive parameter to avoid the denominator by zero, taken as  $10^{-6}$ in the computation as the WENO-XS scheme \cite{Xing-Shu-2005JCP}, unless otherwise stated.
$\beta_\ell$ are the smoothness indicators which measure how smooth the polynomials $p_{\ell}(x),~\ell=1,2,3$ are in the target cell $I_i$.
The explicit formulas are given by
\begin{align*}\small
\beta_1=(\Delta x\phi_{i})^2
&+\frac{13}{3}\big( 3(f_{i-1}- f_{i}) + \Delta x(\phi_{i-1}+2\phi_{i})\big)^2
\\&+\frac{781}{20}\big( 2(f_{i-1}- f_{i}) + \Delta x(\phi_{i-1}+\phi_{i})\big)^2,\\
\beta_2=(\Delta x\phi_{i})^2
&+\frac{13}{3}\big( 3(f_{i}- f_{i+1}) + \Delta x(\phi_{i+1}+2\phi_{i})\big)^2
\\&+\frac{781}{20}\big( 2(f_{i}- f_{i+1}) + \Delta x(\phi_{i+1}+\phi_{i})\big)^2,\\
\beta_3=(\Delta x\phi_{i})^2
&+\frac{13}{12}(f_{i-1}-2 f_i+ f_{i+1})^2
\\&+\frac{781}{80}\big( f_{i-1}- f_{i+1} + 2\Delta x\phi_{i}\big)^2.
\end{align*}

\textbf{Step 2.}
Discretize the spatial derivatives $\big(\tilde{\bm{b}}_x,\widetilde{\bm{B}}_x\big)$ in source terms as
\begin{equation}\label{swe-form-1d-Sdx}
\begin{split}
&\tilde{\bm{b}}_x\approx \frac{\hat{\tilde{\bm{b}}}_{i+\frac12}-\hat{\tilde{\bm{b}}}_{i-\frac12}}{\Delta x},\quad
\widetilde{\bm{B}}_x\approx \frac{\widehat{\widetilde{\bm{B}}}_{i+\frac12}-\widehat{\widetilde{\bm{B}}}_{i-\frac12}}{\Delta x},
\end{split}
\end{equation}
where $\hat{\tilde{\bm{b}}}_{i+\frac12}$ and $\widehat{\widetilde{\bm{B}}}_{i+\frac12}$ are the fifth-order approximation of the values $\tilde{\bm{b}}(x_{i+\frac12})$ and $\widetilde{\bm{B}}(x_{i+\frac12})$ on the cell of $I_{i}$, respectively.
To obtain the well-balanced property, we split the derivative terms of the source terms \eqref{swe-form-1d-h} as the following forms,
\begin{equation}
\begin{split}
&\tilde{\bm{b}} =\tilde{\bm{b}}^{+} + \tilde{\bm{b}}^{-},\quad \quad\tilde{\bm{b}}^{\pm} = \frac{1}{2}\bm{b} ,
\\&\widetilde{\bm{B}} = \widetilde{\bm{B}}^{+} + \widetilde{\bm{B}}^{-},\quad
\widetilde{\bm{B}}^{\pm}=\frac{1}{2}\widetilde{\bm{B}}.
\end{split}
\end{equation}
Then we use the same HWENO approximation, the same local characteristic decomposition and the same nonlinear weights of  $\big(\widehat{\mathbf{F}}^{+}_{i+\frac12},\widehat{\mathbf{\Phi}}^{+}_{i+\frac12}\big)$ and $\big(\widehat{\mathbf{F}}^{-}_{i+\frac12},\widehat{\mathbf{\Phi}}^{-}_{i+\frac12}\big)$ to approximate $\big(\hat{\tilde{\bm{b}}}^{+}_{i+\frac12},\widehat{\widetilde{\bm{B}}}^{+}_{i+\frac12}\big)$ and $\big(\hat{\tilde{\bm{b}}}^{-}_{i+\frac12},\widehat{\widetilde{\bm{B}}}^{-}_{i+\frac12}\big)$, respectively.
Thus, the semi-discrete finite difference HWENO scheme is, for any $i = 1,...,N_x$
\begin{equation}\label{swe-form-1d-semi}
\begin{cases}
\frac{d}{d t} \bm{U}_i =\mathbf{R}_i~= - g\eta_i \frac{\hat{\tilde{\bm{b}}}_{i+\frac12}-\hat{\tilde{\bm{b}}}_{i-\frac12}}{\Delta x}
-\frac{\widehat{\mathbf{F}}_{i+\frac12}-\widehat{\mathbf{F}}_{i-\frac12}}{\Delta x},
\\
\frac{d}{d t} \bm{V}_i=\mathbf{R}_{x,i}=-gH_i\frac{\hat{\tilde{\bm{b}}}_{i+\frac12}-\hat{\tilde{\bm{b}}}_{i-\frac12}}{\Delta x}- g\eta_i \frac{\widehat{\widetilde{\bm{B}}}_{i+\frac12}-\widehat{\widetilde{\bm{B}}}_{i-\frac12}}{\Delta x}
-\frac{\widehat{\mathbf{\Phi}}_{i+\frac12}-\widehat{\mathbf{\Phi}}_{i-\frac12}}{\Delta x},
\end{cases}
\end{equation}
where
\begin{equation}
\begin{split}
&\widehat{\mathbf{F}}_{i+\frac12}  = \widehat{\mathbf{F}}^{+}_{i+\frac12}+\widehat{\mathbf{F}}^{-}_{i+\frac12},\quad
\widehat{\mathbf{\Phi}}_{i+\frac12} =\widehat{\mathbf{\Phi}}^{+}_{i+\frac12}+\widehat{\mathbf{\Phi}}^{-}_{i+\frac12},
\\
&\hat{\tilde{\bm{b}}}_{i+\frac12} = \hat{\tilde{\bm{b}}}^{+}_{i+\frac12}+\hat{\tilde{\bm{b}}}^{-}_{i+\frac12},\quad
\widehat{\widetilde{\bm{B}}}_{i+\frac12} = \widehat{\widetilde{\bm{B}}}^{+}_{i+\frac12}+\widehat{\widetilde{\bm{B}}}^{-}_{i+\frac12}.
\end{split}
\end{equation}
Obviously, the residue of system \eqref{swe-form-1d-semi} equals to zero when the still water steady-state solution is reached, i.e., the semi-discrete scheme \eqref{swe-form-1d-semi} is well-balanced.

\subsection{Time discretization and limiter}
\label{sec:limiter-1d}

In this section, we consider the explicit third-order strong stability-preserving (SSP) Runge-Kutta scheme to discretize  \eqref{swe-form-1d-semi} in time, as
\begin{equation}
\label{RK}
\begin{cases}
\begin{cases}
\bm{U}^{(1)}_i =\bm{U}^n_i~~~~
+ \Delta t \mathbf{R}_{i}(\bm{U}^{n},\bm{V}^{n}),\\
\bm{V}^{(1)}_i =\bm{V}^{n,mod}_i
+ \Delta t \mathbf{R}_{x,i}(\bm{U}^{n},\bm{V}^{n}),
\end{cases}\\
\begin{cases}
\bm{U}^{(2)}_i =\frac{3}{4}\bm{U}^n_i~~~~+
\frac{1}{4}\big( \bm{U}^{(1)}_i~~~~~+ \Delta t \mathbf{R}_{i}(\bm{U}^{(1)},\bm{V}^{(1)})\big),\\
\bm{V}^{(2)}_i =\frac{3}{4}\bm{V}^{n,mod}_i +
\frac{1}{4}\big( \bm{V}^{(1),mod}_i + \Delta t  \mathbf{R}_{x,i}(\bm{U}^{(1)},\bm{V}^{(1)})\big),
\end{cases}\\
\begin{cases}
\bm{U}^{n+1}_i =\frac{1}{3}\bm{U}^n_i~~~~+
\frac{2}{3}\big(\bm{U}^{(2)}_i~~~~~+ \Delta t \mathbf{R}_{i}(\bm{U}^{(2)},\bm{V}^{(2)})\big),\\
\bm{V}^{n+1}_i =\frac{1}{3}\bm{V}^{n,mod}_i +
\frac{2}{3}\big( \bm{V}^{(2),mod}_i + \Delta t  \mathbf{R}_{x,i}(\bm{U}^{(2)},\bm{V}^{(2)})\big),
\end{cases}
\end{cases}
\end{equation}
where $\bm{V}^{mod}$ is the modified value for the derivative $\bm{V}$, and needs to be modified in each intermediate step.
This procedure is used to control spurious oscillations, and the detail can be seen in \cite{ZhaoZhuang-2020JSC-FD}.

We now briefly describe the HWENO limiter procedure to modify $\bm{V}_i$ and obtain $\bm{V}_i^{mod}$ finally. Similarly as in the spatial reconstruction, the HWENO limiter is also performed on local characteristic directions. For simplicity, we here only describe the procedure in scalar case by assuming that $\nu_i$ is the derivative of $\mu_i$. We use the same stencils $\mathcal{T}_1$, $\mathcal{T}_2$, and $\mathcal{T}_3$ as that in the spatial HWENO reconstruction, then, apply the Hermite interpolation on these stencils, and obtain three quadratic polynomials $p_1(x),p_2(x),p_3(x)$ such that
\begin{equation*}
\begin{split}
&p_1(x): ~~p_1(x_{i+\ell})=\mu_{i+\ell}, \quad \ell=-1,0,\qquad p'_1(x_{i-1})=\nu_{i-1},\\
&p_2(x): ~~p_2(x_{i+\ell})=\mu_{i+\ell}, \quad \ell=0,1,\qquad\quad p'_2(x_{i+1})=\nu_{i+1},\\
&p_3(x): ~~p_3(x_{i+\ell})=\mu_{i+\ell}, \quad \ell=-1,0,1.
\end{split}
\end{equation*}
After, we take the derivative of $p_{\ell}(x),~\ell=1,2,3$, with respect to $x$, and evaluate their values at $x_i$,  obtaining
\begin{equation}\label{step1-dxi}
\begin{split}
p'_1(x_i)&=\frac{2}{\Delta x}\big(\mu_i-\mu_{i-1}\big) - \nu_{i-1},\\
p'_2(x_i)&=\frac{2}{\Delta x}\big(\mu_{i+1}-\mu_{i}\big)- \nu_{i+1},\\
p'_3(x_i)&=\frac{1}{2\Delta x}\big(\mu_{i+1}-\mu_{i-1}\big).
\end{split}
\end{equation}
Finally, the modified derivative $ \nu^{mod}_i$ is defined as
\begin{equation}\label{limiter-1d}
 \nu^{mod}_i =  \omega^L_1 p'_1(x_i)+ \omega^L_2 p'_2(x_i) +\omega^L_3 p'_3(x_i),
\end{equation}
where the nonlinear weights $\omega^L_{\ell},~\ell=1,2,3$ are computed as similar as the above procedure \eqref{GHYZZ} in the reconstruction. The linear weights are $d_1=\frac14$, $d_2=\frac14$, and $d_3=\frac12$, and the
smoothness indicators $\beta^L_{\ell},~\ell=1,2,3$ are
\begin{align*}\small
    \beta^L_1&=(2\mu_{i}-2\mu_{i-1}-  \nu_{i-1}\Delta x)^2+\frac{13}{3}(\mu_{i}- \mu_{i-1}-\nu_{i-1}\Delta x)^2,\\
    \beta^L_2&=(2\mu_{i+1}-2\mu_i-\nu_{i+1}\Delta x)^2+\frac{13}{3}(\mu_{i+1}-\mu_i-\nu_{i+1}\Delta x)^2,\\
   \beta^L_3&=\frac{1}{4}(\mu_{i+1}-\mu_{i-1})^2+\frac{13}{12}(\mu_{i-1}-2 \mu_i+ \mu_{i+1})^2.
\end{align*}

It is not difficult to know that $|\bm{V}_i^{mod}-\bm{V}_i|=\mathcal{O}(\Delta x ^4)$ for smooth solutions,
and $\bm{V}_i^{mod}=\bm{V}_i=0$ exactly for the still water steady-state. That is to say, the limiter procedure maintains the properties of fifth-order accuracy and well-balance.

\section{Well-balanced HWENO scheme for 2D SWEs}
\label{sec:WBFDHWENO-2d}

In this section, we present a well-balanced fifth-order finite difference HWENO scheme for the 2D SWEs. It is straightforward to extend the scheme into 1D to 2D by dimension-by-dimension manner.
The SWEs in 2D are read as
\begin{equation}\label{swe-form-2d}
\frac{\partial}{\partial t}
\begin{bmatrix*}[l]
  h\\
  hu\\
  hv\\
\end{bmatrix*}
+\frac{\partial}{\partial x}
\begin{bmatrix*}[c]
 hu\\
 hu^2+\frac{1}{2}gh^2\\
huv\\
\end{bmatrix*}
+\frac{\partial}{\partial y}
\begin{bmatrix*}[c]
 w\\
huv\\
hv^2+\frac{1}{2}gh^2\\
\end{bmatrix*}
=
\begin{bmatrix*}[c]
  0\\
  -ghb_x\\
  - ghb_y\\
\end{bmatrix*},
\end{equation}
where $h(x,y,t)\geq0$ is the depth of water, $(hu, hv)$ are the discharges,
$(u,v)$ are the velocities,
$b = b(x,y)$ is the bottom topography assumed to be a given time-independent function, and
$g$ is the gravitation acceleration. The still water steady-state solution of $\eqref{swe-form-2d}$ is
\begin{equation}\label{st-2d}
\eta=h+b = C,\quad m=hu=0, \quad w=hv=0.
\end{equation}
Let $\tilde{\bm{b}} = (0,b,0)^{T}$ and $\bar{\bm{b}} = (0,0,b)^{T}$, and
rewrite \eqref{swe-form-2d} into a pre-balanced form with the equilibrium variable $\bm{U} = (\eta = h+b,m=hu,w=hv)^{T}$ as
\begin{equation}\label{swe-form-2d-e}
\frac{\partial \bm{U}}{\partial t}
+\frac{\partial}{\partial x}
\mathbf{F}_1(\bm{U},b)
+\frac{\partial}{\partial y}
\mathbf{F}_2(\bm{U},b)
=\bm{S}(\bm{U},b),
\end{equation}
where
\begin{equation}
\begin{split}
&\mathbf{F}_1(\bm{U},b) = \begin{bmatrix*}[c]
 m\\
  \frac{m^2}{\eta-b}+\frac{1}{2}g\big(\eta^2-2\eta b)\\
  \frac{mw}{\eta-b}
\end{bmatrix*},
\quad
\mathbf{F}_2(\bm{U},b)=\begin{bmatrix*}[c]
 w\\
   \frac{mw}{\eta-b}\\
   \frac{w^2}{\eta-b}+\frac{1}{2}g\big(\eta^2-2\eta b)
\end{bmatrix*},
\\&\bm{S} = \begin{bmatrix*}[c]
  0\\
  -g\eta b_x\\
   -g\eta b_y
\end{bmatrix*} =-g\eta \tilde{\bm{b}}_x-g\eta \bar{\bm{b}}_y.\\
\end{split}
\end{equation}

To construct an Hermite WENO scheme, we take the partial
derivative w.r.t. the variables $x$ and $y$ on both sides of \eqref{swe-form-2d-e}, respectively, and denote
\begin{align*}
& \bm{V} = (H_1,M_1,W_1)^{T}=(\eta_x,m_x,w_x)^{T} = \bm{U}_x,
\\&\bm{Q} = (H_2,M_2,W_2)^{T}=(\eta_y,m_y,w_y)^{T}= \bm{U}_y,
\\&
\widetilde{\bm{B}}_1= (0,B_1,0)^{T}= (0,b_x,0)^{T}=\tilde{\bm{b}}_x ,
\quad \widetilde{\bm{B}}_2 = (0,B_2,0)^{T}= (0,b_y,0)^{T}=\tilde{\bm{b}}_y,
\\&
\bar{\bm{B}}_1= (0,0, B_1)^{T}= (0,0, b_x)^{T}=\bar{\bm{b}}_x,
\quad \bar{\bm{B}}_2= (0,0, B_2)^{T}= (0,0, b_y)^{T}=\bar{\bm{b}}_y.
\end{align*}
Then we have
\begin{equation}\label{swe-form-2d-h}
\begin{cases}
\frac{\partial \bm{U}}{\partial t}
+\frac{\partial\mathbf{F}_1}{\partial x}
+\frac{\partial\mathbf{F}_2}{\partial y}
=-g\eta \tilde{\bm{b}}_x-g\eta \bar{\bm{b}}_y,
\\
\frac{\partial \bm{V}}{\partial t}
+\frac{\partial\mathbf{\Phi}_1}{\partial x}
+\frac{\partial\mathbf{\Phi}_2}{\partial y}
=
 -g\eta(\widetilde{\bm{B}}_1)_x-g\eta(\bar{\bm{B}}_1)_y -gH_1\tilde{\bm{b}}_x-gH_1\bar{\bm{b}}_y,
\\
\frac{\partial \bm{Q}}{\partial t}
+\frac{\partial\mathbf{\Psi}_1}{\partial x}
+\frac{\partial\mathbf{\Psi}_2}{\partial y}
=
 -g\eta(\widetilde{\bm{B}}_2)_x-g\eta(\bar{\bm{B}}_2)_y -gH_2\tilde{\bm{b}}_x-gH_2\bar{\bm{b}}_y,
\end{cases}
\end{equation}
where
\begin{align*}
\mathbf{\Phi}_1(\bm{U},b;\bm{V},B_1) 
&=\frac{\partial \mathbf{F}_1}{\partial \bm{U}}\bm{V} +\frac{\partial\mathbf{F}_1}{\partial \tilde{\bm{b}}}\widetilde{\bm{B}}_1,\ \ \mathbf{\Phi}_2(\bm{U},b;\bm{V},B_1) 
=\frac{\partial \mathbf{F}_2}{\partial \bm{U}}\bm{V} +\frac{\partial\mathbf{F}_2}{\partial \bar{\bm{b}}}\bar{\bm{B}}_1,\\
%
\mathbf{\Psi}_1(\bm{U},b;\bm{Q},B_2) 
&=\frac{\partial \mathbf{F}_1}{\partial \bm{U}}\bm{Q} +\frac{\partial\mathbf{F}_1}{\partial \tilde{\bm{b}}}\widetilde{\bm{B}}_2,\ \
\mathbf{\Psi}_2(\bm{U},b;\bm{Q},B_2) 
=\frac{\partial \mathbf{F}_2}{\partial \bm{U}}\bm{Q} +\frac{\partial\mathbf{F}_2}{\partial \bar{\bm{b}}}\bar{\bm{B}}_2.
%
\end{align*}
For the still water steady-state solution of $\eqref{swe-form-2d}$, we have
\begin{equation}\label{s-2d}
\begin{split}
&\eta=h+b = Const,\qquad m=hu=0, \qquad w=hv=0,\\
&H_1 = \eta_x=0,\qquad\qquad M_1=m_x=0,\qquad W_1= w_x = 0,
\\&H_2 = \eta_y=0,\qquad\qquad M_2=m_y=0,\qquad W_2= w_y = 0.
\end{split}
\end{equation}

Similarly as the 1D case, we would like to find a linear scheme by using linear finite difference operators to approximate all derivatives for the 2D system \eqref{swe-form-2d-h}.

\begin{prop}
Linear scheme for the 2D system \eqref{swe-form-2d-h} satisfying the still water steady-state solution \eqref{st-2d} can maintain the well-balance property.
\end{prop}

Note that the governing equations \eqref{swe-form-2d-h} have a part of similar expressions with \eqref{swe-form-1d-h}.
Hence, we can apply the same procedures of 1D case to discretize the part of fluxes with similar expressions. More explicitly, we use the procedure in {\bf Step 1} to discretize $(\frac{\partial \mathbf{F}_1} {\partial x}, \frac{\partial \mathbf{\Phi}_1} {\partial x}\big)$ and $\big(\frac{\partial \mathbf{F}_2} {\partial y}, \frac{\partial \mathbf{\Psi}_2} {\partial y} \big)$ along $x$- and $y$-directions, respectively; and, use the procedure in {\bf Step 2} to discretize the source terms
$\big( \tilde{\bm{b}}_x,(\widetilde{\bm{B}}_1)_x \big)$ and $\big( \bar{\bm{b}}_y,(\bar{\bm{B}}_2)_y \big)$ along $x$- and $y$-directions, respectively.
For the mixed derivative terms $\frac{\partial \mathbf{\Psi}_1} {\partial x}$, $\frac{\partial \mathbf{\Phi}_2} {\partial y}$, $(\widetilde{\bm{B}}_2)_x$ and $(\bar{\bm{B}}_1)_y$, we directly use the linear approximation in each component (without any local characteristic decomposition) as
\begin{equation}\label{mixed}
\begin{split}
&(\widehat{\mathbf{\Psi}}_1)_{i+\frac12,j}=
-\frac{1}{12}(\mathbf{\Psi}_1)_{i-1,j}+\frac{7}{12}(\mathbf{\Psi}_1)_{i,j}
+\frac{7}{12}(\mathbf{\Psi}_1)_{i+1,j}-\frac{1}{12}(\mathbf{\Psi}_1)_{i+2,j},
\\&(\widehat{\mathbf{\Phi}}_2)_{i,j+\frac12}=
-\frac{1}{12}(\mathbf{\Phi}_2)_{i,j-1}+\frac{7}{12}(\mathbf{\Phi}_2)_{i,j}
+\frac{7}{12}(\mathbf{\Phi}_2)_{i,j+1}-\frac{1}{12}(\mathbf{\Phi}_2)_{i,j+2},
\end{split}
\end{equation}
and the approximation of $(\widehat{\widetilde{\bm{B}}}_2)_{i+\frac12,j}$ and $(\widehat{\bar{\bm{B}}}_1)_{i,j+\frac12}$ is same to \eqref{mixed} for $(\widehat{\mathbf{\Psi}}_1)_{i+\frac12,j}$ and $(\widehat{\mathbf{\Phi}}_2)_{i,j+\frac12}$, respectively.

Then, the semi-discrete well-balanced finite difference HWENO scheme of \eqref{swe-form-2d-h} is
\begin{equation}\label{swe-form-2d-dis}
\begin{cases}
\frac{d}{d t}\bm{U}_{i,j}
= &
-\frac{(\hat{\mathbf{F}}_{1})_{i+\frac12,j}-(\hat{\mathbf{F}}_{1})_{i-\frac12,j}}{\Delta x}
-\frac{(\hat{\mathbf{F}}_{2})_{i,j+\frac12}-(\hat{\mathbf{F}}_{2})_{i,j-\frac12}}{\Delta y}
\\&+ g\eta_{i,j}\frac{\hat{\tilde{\bm{b}}}_{i+\frac12,j}-\hat{\tilde{\bm{b}}}_{i-\frac12,j}}{\Delta x}+ g\eta_{i,j}\frac{\hat{\bar{\bm{b}}}_{i,j+\frac12}-\hat{\bar{\bm{b}}}_{i,j-\frac12}}{\Delta y},
\\
\frac{d}{d t}\bm{V}_{i,j} =&
-\frac{(\hat{\mathbf{\Phi}}_{1})_{i+\frac12,j}-(\hat{\mathbf{\Phi}}_{1})_{i-\frac12,j}}{\Delta x}-\frac{(\hat{\mathbf{\Phi}}_{2})_{i,j+\frac12}-(\hat{\mathbf{\Phi}}_{2})_{i,j-\frac12}}{\Delta y}
\\&
+ g\eta_{i,j}\frac{(\widehat{\widetilde{\bm{B}}}_1)_{i+\frac12,j}
                -\widehat{\widetilde{\bm{B}}}_1)_{i-\frac12,j}}{\Delta x}
+g\eta_{i,j}\frac{(\widehat{\bar{\bm{B}}}_1)_{i,j+\frac12}
-(\widehat{\bar{\bm{B}}}_1)_{i,j-\frac12}}{\Delta y}
\\&+ g(H_1)_{i,j}\frac{\hat{\tilde{\bm{b}}}_{i+\frac12,j}-\hat{\tilde{\bm{b}}}_{i-\frac12,j}}{\Delta x}+ g(H_1)_{i,j}\frac{\hat{\bar{\bm{b}}}_{i,j+\frac12}-\hat{\bar{\bm{b}}}_{i,j-\frac12}}{\Delta y},\\
\frac{d}{d t}\bm{Q}_{i,j} =&
-\frac{(\hat{\mathbf{\Psi}}_{1})_{i+\frac12,j}-(\hat{\mathbf{\Psi}}_{1})_{i-\frac12,j}}{\Delta x}-\frac{(\hat{\mathbf{\Psi}}_{2})_{i,j+\frac12}-(\hat{\mathbf{\Psi}}_{2})_{i,j-\frac12}}{\Delta y}
\\&
+ g\eta_{i,j}\frac{(\widehat{\widetilde{\bm{B}}}_2)_{i+\frac12,j}
                -\widehat{\widetilde{\bm{B}}}_2)_{i-\frac12,j}}{\Delta x}
+g\eta_{i,j}\frac{(\widehat{\bar{\bm{B}}}_2)_{i,j+\frac12}
-(\widehat{\bar{\bm{B}}}_2)_{i,j-\frac12}}{\Delta y}
\\&+ g(H_2)_{i,j}\frac{\hat{\tilde{\bm{b}}}_{i+\frac12,j}-\hat{\tilde{\bm{b}}}_{i-\frac12,j}}{\Delta x}+ g(H_2)_{i,j}\frac{\hat{\bar{\bm{b}}}_{i,j+\frac12}-\hat{\bar{\bm{b}}}_{i,j-\frac12}}{\Delta y}.
\end{cases}
\end{equation}

\vspace{10pt}

Similarly as in one dimension, we use the explicit third-order SSP Runge-Kutta scheme \eqref{RK} to discretize \eqref{swe-form-2d-dis}, and apply the HWENO limiter (cf. \eqref{limiter-1d}) to control the derivatives $\bm{V}_{i,j}$ and $\bm{Q}_{i,j}$ in the time discretization step by a dimension-by-dimensional manner.

\section{Numerical results}
\label{sec:numerical}

In this section we present the numerical results of the well-balanced fifth-order finite difference HWENO scheme described in the previous sections for the one- and two-dimensional shallow water equations with non-flat bottom topography. The CFL number is taken as $0.6$, and the gravitation constant is taken as $g=9.812$ in the computation.
Since its analytical form is unavailable, for comparisons, we take the numerical solution obtained by the fifth-order finite difference WENO-XS scheme \cite{Xing-Shu-2005JCP} with $N_x=3000$ in 1D case and with $N_x=N_y=1600$ in 2D case as the reference, unless otherwise stated.

\begin{example}\label{test5-1d}
(The accuracy test for the 1D SWEs over a sinusoidal hump.)
\end{example}
This example is used to verify the high order accuracy and efficiency of the proposed HWENO scheme.
The bottom topography is
\begin{equation*}
b(x)=\sin^2(\pi x), \quad x \in [0,1].
\end{equation*}
We use periodic boundary conditions for all unknown variables. The initial conditions are given as
\begin{equation*}
h(x,0)=5+e^{\cos(2\pi x)}, \quad
hu(x,0)=\sin\big(\cos(2\pi x)\big) .
\end{equation*}
We compute the solution up to $t=0.1$ when the solution is still smooth.
A reference solution is obtained using the fifth-order finite difference WENO-XS scheme \cite{Xing-Shu-2005JCP} with $N_x=25600$, and treat this reference solution as the exact solution in computing the numerical errors.
The error of $L^1$ and $L^\infty$ norm of the proposed HWENO scheme for the water depth $h$ and water discharge $hu$ are plotted in Fig.~\ref{Fig:test5-1d-h-hu}.
It can be seen that the scheme has the expected fifth-order convergence in both $L^1$ and $L^\infty$ norm. Moreover, the figures show that the HWENO scheme produces a  smaller error than the WENO-XS scheme \cite{Xing-Shu-2005JCP} on the same number of elements. It is worth pointing out that the proposed fifth-order HWENO scheme only needs a compact three-point stencil while the fifth-order WENO-XS scheme needs a five-point stencil in the reconstruction. Thus, the HWENO scheme is more accurate and compact than the WENO-XS scheme.

To show the efficiency of the proposed HWENO scheme, we present  the error of $L^1$ norm against the CPU time for the water depth $h$ and water discharge $hu$ in Fig.~\ref{Fig:test5-1d-cpu}.
One can see that the HWENO scheme is more efficient than  the WENO-XS scheme \cite{Xing-Shu-2005JCP} as the former obtains a smaller error  for a fixed CPU time.

\begin{figure}[H]
\centering
\subfigure[$h$]{
\includegraphics[width=0.45\textwidth,trim=20 10 30 30,clip]{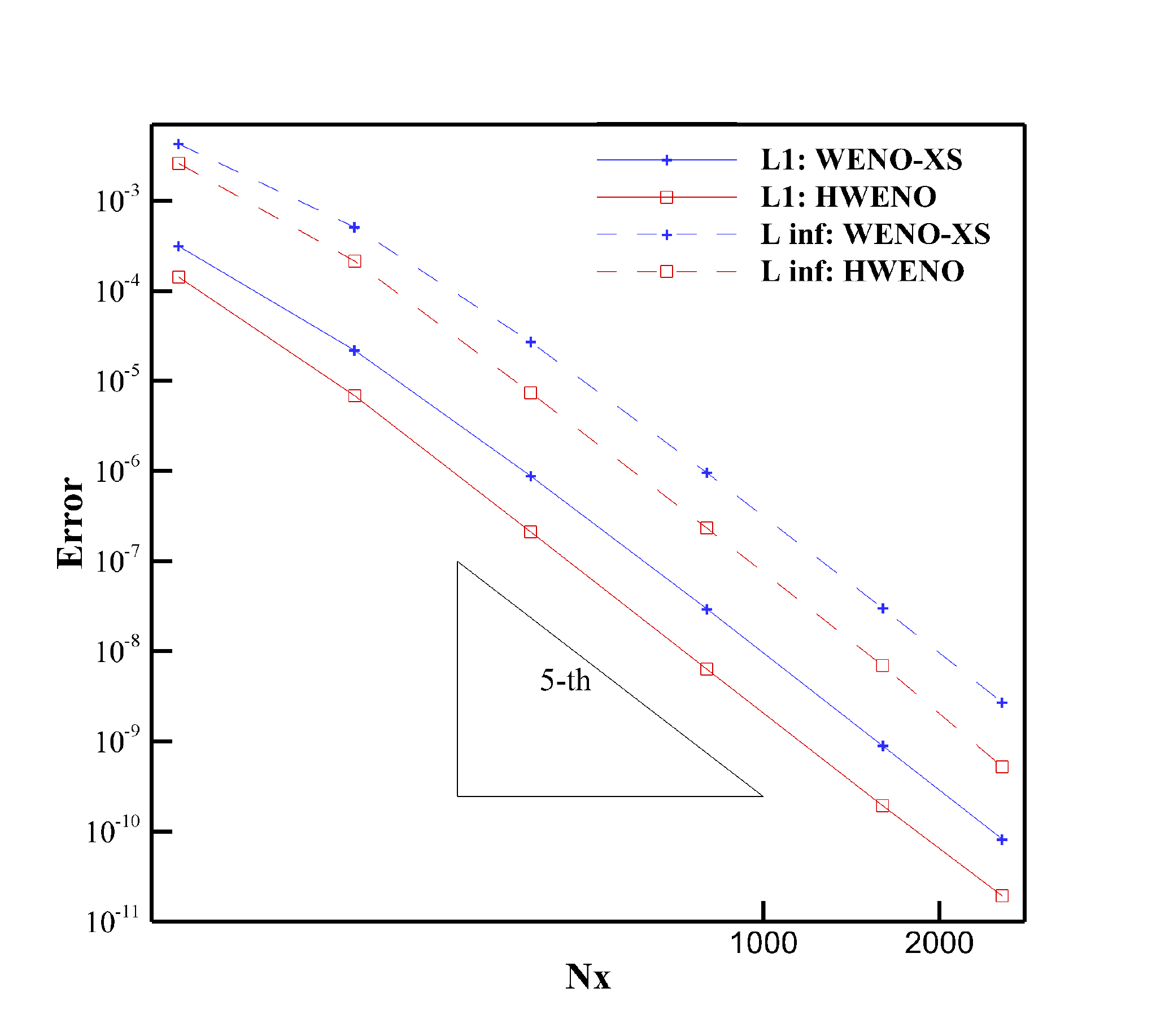}}
\subfigure[$hu$]{
\includegraphics[width=0.45\textwidth,trim=20 10 30 30,clip]{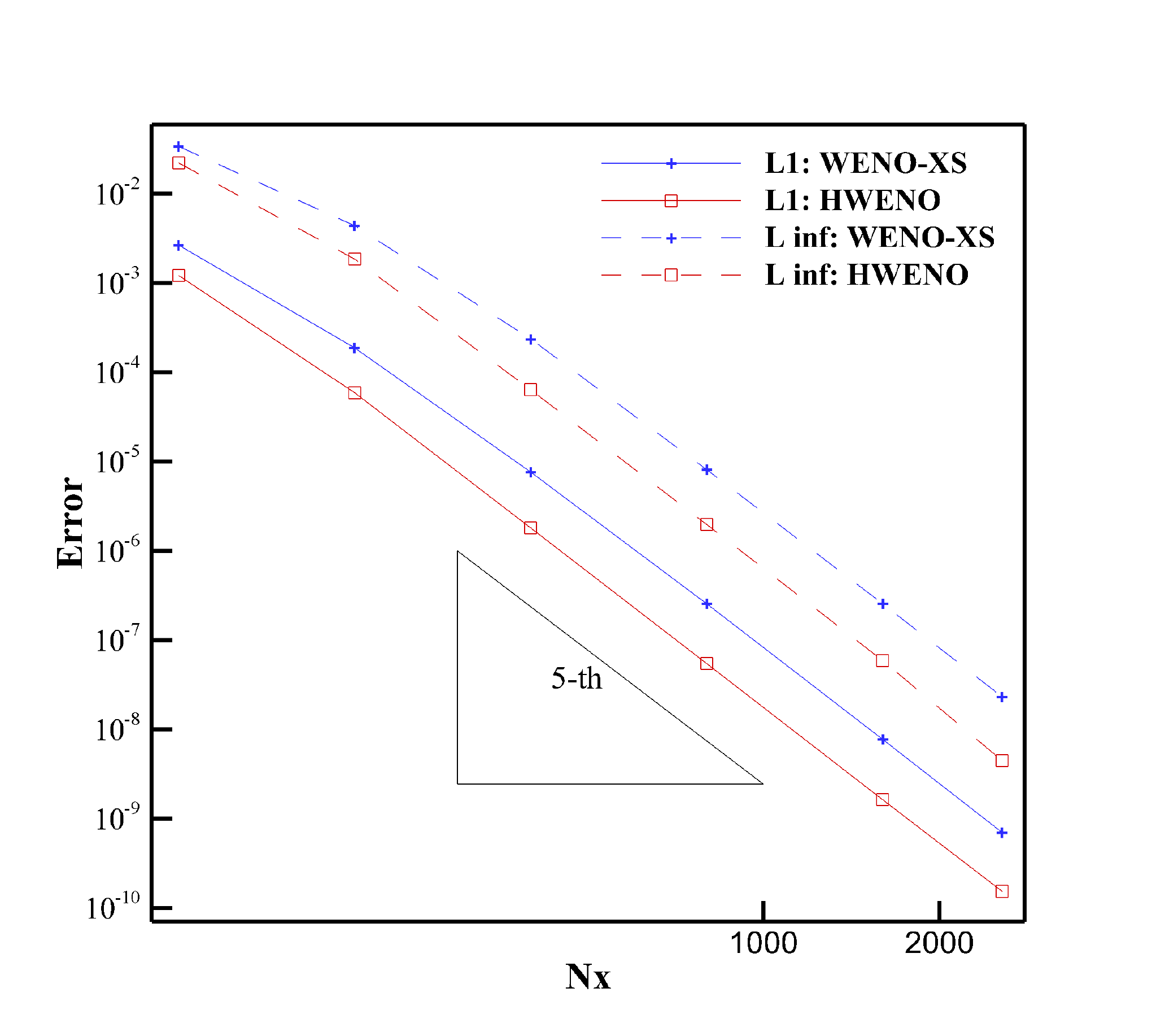}}
\caption{Example \ref{test5-1d}. The error of $L^1$ and $L^\infty$ norm for the water depth $h$ and water discharge $hu$.}
\label{Fig:test5-1d-h-hu}
\end{figure}

\begin{figure}[H]
\centering
\subfigure[$h$]{
\includegraphics[width=0.45\textwidth,trim=20 10 30 30,clip]{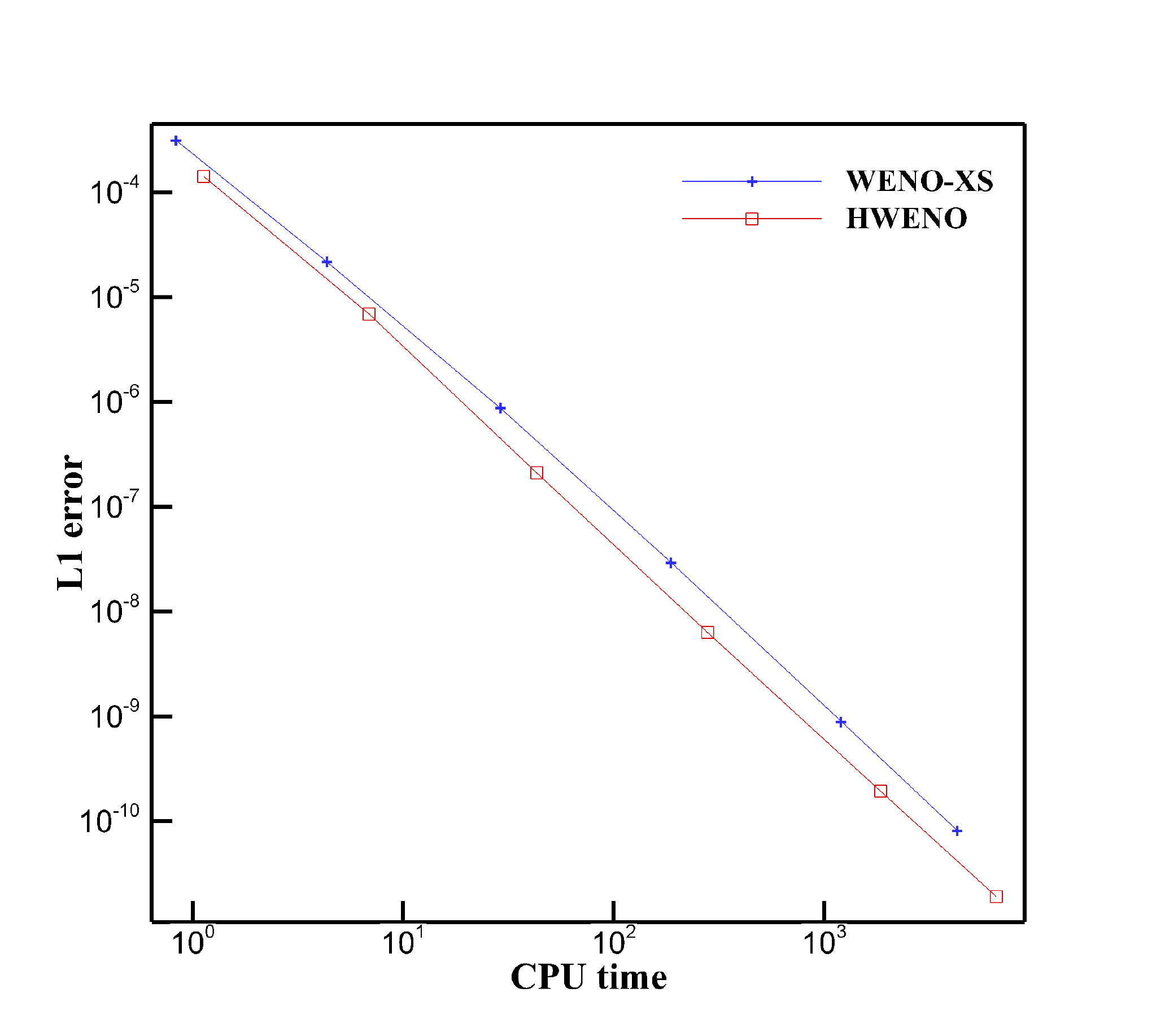}}
\subfigure[$hu$]{
\includegraphics[width=0.45\textwidth,trim=20 10 30 30,clip]{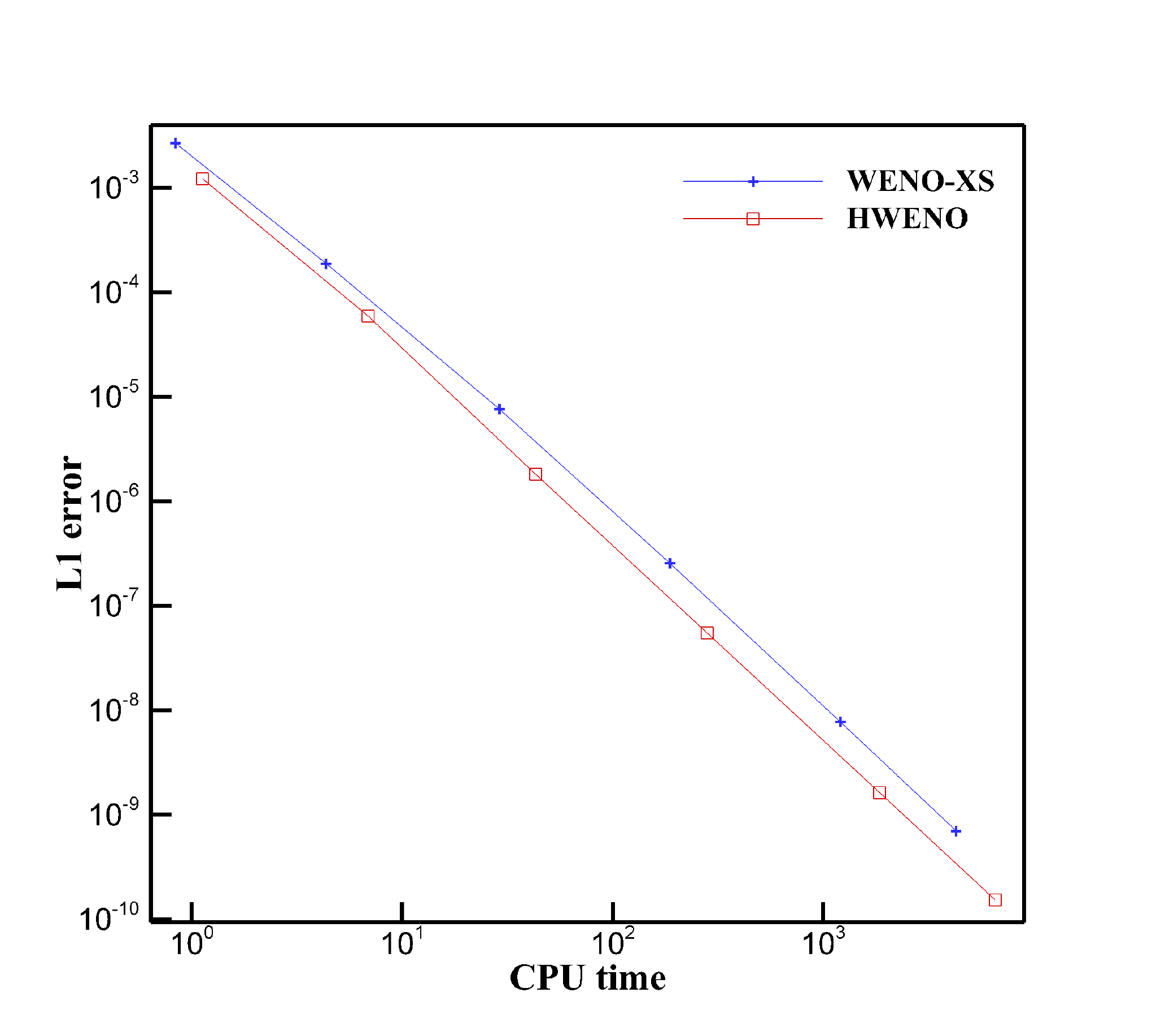}}
 \caption{Example \ref{test5-1d}. The error of $L^1$ norm against the CPU time.}
\label{Fig:test5-1d-cpu}
\end{figure}

\begin{example}\label{test1+2-1d}
(The lake-at-rest steady-state flow test for the 1D SWEs.)
\end{example}
In this example we consider a lake-at-rest steady-state flow over a smooth and a discontinuous bottom topographies
to verify the well-balance property of the HWENO scheme. The bottom topographies are given by
\begin{align}
& b(x)=5 e^{-\frac25(x-5)^2},  \quad x \in [0,10],
\label{b-1}
\\
& b(x)=
\begin{cases}
4,&  x \in [4, 8], \\
0,&  x \in [0,4) \cup (8, 10].
\end{cases}
\label{b-2}
\end{align}
The initial data is at the lake-at-rest steady-state $u=0, ~h+b=10$.
This still water steady-state solution should be preserved exactly if the HWENO scheme is well-balanced.

To show that the well-balance property is attained up to the level of round-off error, we computed the solution up to $t=0.5$ using single, double and quadruple precisions with $N_x=200$.
The $L^1$ and $L^\infty$ error for $h+b$ and $hu$ are presented in Table~\ref{tab:test1-1d-error} for the smooth bottom topography.
Similar results for the discontinuous bottom topography are shown in Table~\ref{tab:test2-1d-error}.
We can clearly see that the $L^1$ and $L^\infty$ error for $h+b$ and $hu$ are at the level of round-off errors for different precisions, and it clearly shows that the HWENO scheme is well-balanced.

\begin{table}[H]
\centering
\caption{Example \ref{test1+2-1d}.
Well-balanced test over the smooth bottom topography \eqref{b-1}.}
\label{tab:test1-1d-error}
\medskip
\begin{tabular} {llllll}
 \toprule
Precision    & \multicolumn{2}{c}{$h$} && \multicolumn{2}{c}{$hu$}  \\
   \cline{2-3} \cline{5-6}
 &$L^1$ error & $L^\infty$ error &  & $L^1$ error & $L^\infty$ error \\
\midrule
Single &    9.63E-05 &     1.02E-04 & &    3.09E-05 &     8.66E-05\\
Double &     3.90E-15 &     1.42E-14 &  &   3.43E-14 &     1.39E-13\\
Quadruple &     3.32E-33 &     1.08E-32 &  &   3.35E-32 &     1.35E-31\\
\bottomrule
\end{tabular}
\end{table}

\begin{table}[H]
\centering
\caption{Example \ref{test1+2-1d}.
Well-balanced test over the discontinuous bottom topography \eqref{b-2}.}
\label{tab:test2-1d-error}
\medskip
\begin{tabular} {llllll}
 \toprule
Precision    & \multicolumn{2}{c}{$h$} &  & \multicolumn{2}{c}{$hu$}  \\
\cline{2-3} \cline{5-6}
&$L^1$ error & $L^\infty$ error &  & $L^1$ error & $L^\infty$ error \\
\midrule
Single &     9.12E-05 &     9.82E-05 & &    3.93E-05 &     1.58E-04\\
Double &     3.22E-15 &     1.24E-14 & &    2.68E-14 &     9.88E-14\\
Quadruple &     3.52E-33 &     1.08E-32 & &    2.88E-32 &     1.02E-31\\
\bottomrule
\end{tabular}
\end{table}

\begin{example}\label{test3-1d}
(The perturbed lake-at-rest steady-state flow test for the 1D SWEs.)
\end{example}
This example is used to show the ability of the HWENO scheme to accurately compute small perturbations of a lake-at-rest steady-state flow over non-flat bottom topography.
The bottom topography in this example is taken as
\begin{equation}
\label{B-5}
b(x)=
\begin{cases}
0.25(\cos(10\pi(x-1.5))+1),&  x \in [1.4, 1.6], \\
0,&  x \in [0, 1.4) \cup (1.6, 2],
\end{cases}
\end{equation}
which has a bump in the middle of the physical interval.
The initial conditions are given by
\begin{equation*}
h(x,0)=
\begin{cases}
1-b(x)+\varepsilon,& x \in [1.1, 1.2],\\
1-b(x),& \text{otherwise,}
\end{cases}
\quad  u(x,0)=0,
\end{equation*}
where $\varepsilon$ is a constant for the perturbation magnitude.
We here consider two cases of $\varepsilon=0.2$ and $\varepsilon=10^{-3}$.
The initial conditions for two cases are plotted in Fig.~\ref{Fig:test3-1d-initial}. The initial wave splits into two waves propagating at the characteristic speeds $\pm \sqrt{gh}$.

We compute the solution up to $t=0.2$ when the right wave has
already passed the bottom bump.
In Fig.~\ref{Fig:test3-1d-case1}, we plot the free water surface $h+b$ and discharge $hu$  obtained by the HWENO and WENO-XS schemes with $N_x=200$ for $\varepsilon=0.2$.
The similar results for $\varepsilon=10^{-3}$ are plotted in Fig.~\ref{Fig:test3-1d-case2}.
These figures show that the HWENO scheme is able to capture the waves of large or small pulses.
From the Figs.~\ref{Fig:test3-1d-case1} and \ref{Fig:test3-1d-case2}, we can clearly see that there are no spurious numerical oscillations, verifying the essentially non-oscillatory property of the proposed HWENO scheme,
and the results of the HWENO scheme have a slightly better resolution than those of the WENO-XS scheme.
Note that, for the small perturbation magnitude $\varepsilon=10^{-3}$, we take $\epsilon=10^{-10}$ in formula \eqref{GHYZZ} to calculate the HWENO nonlinear weights, and  $\epsilon$ is also reduced in the WENO-XS scheme \cite{Xing-Shu-2005JCP}.

\begin{figure}[H]
\centering
\subfigure[$\varepsilon =0.2$]{
\includegraphics[width=0.45\textwidth,trim=0 0 20 10,clip]{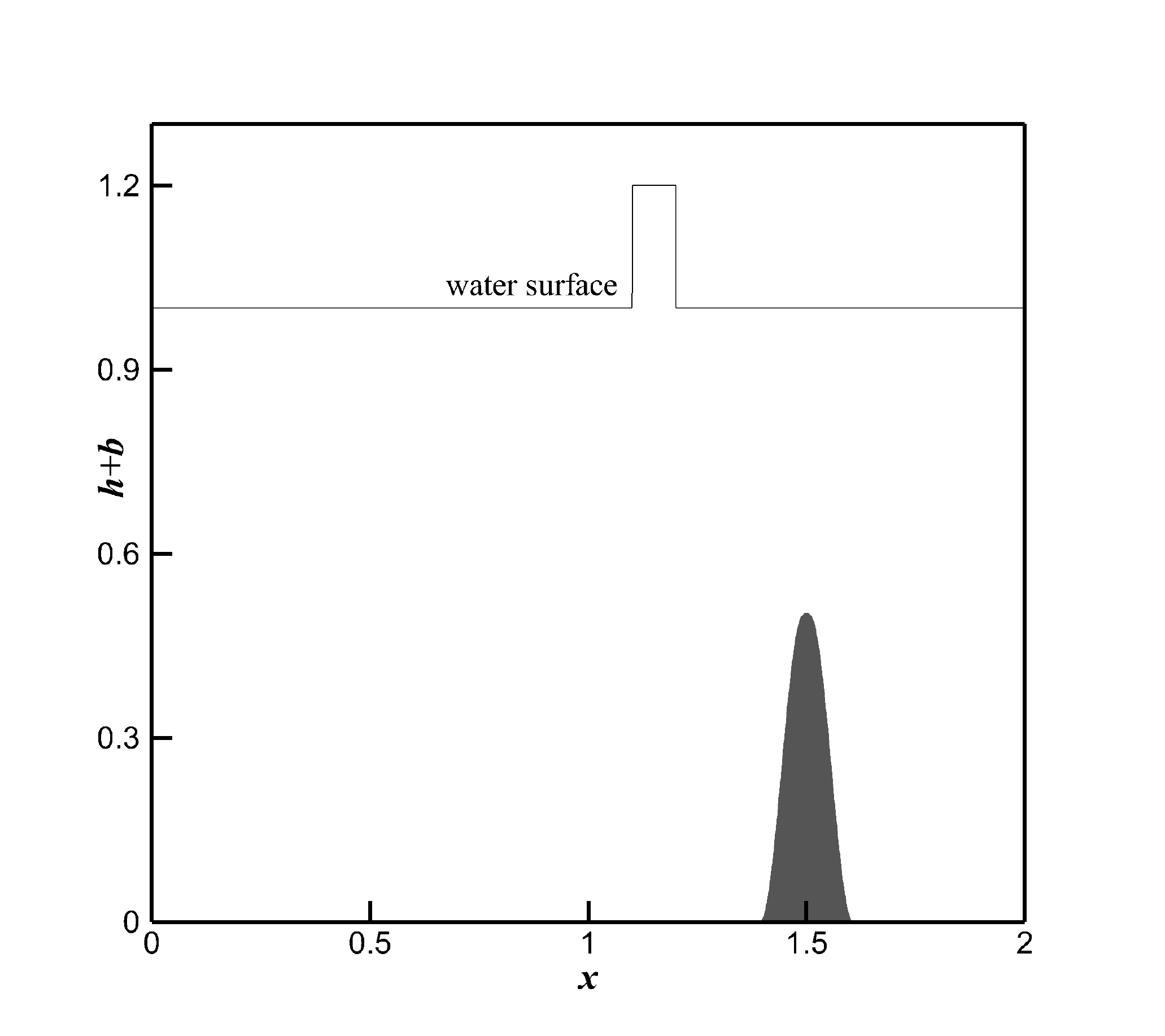}}
\subfigure[$\varepsilon =10^{-3}$]{
\includegraphics[width=0.45\textwidth,trim=0 0 20 10,clip]{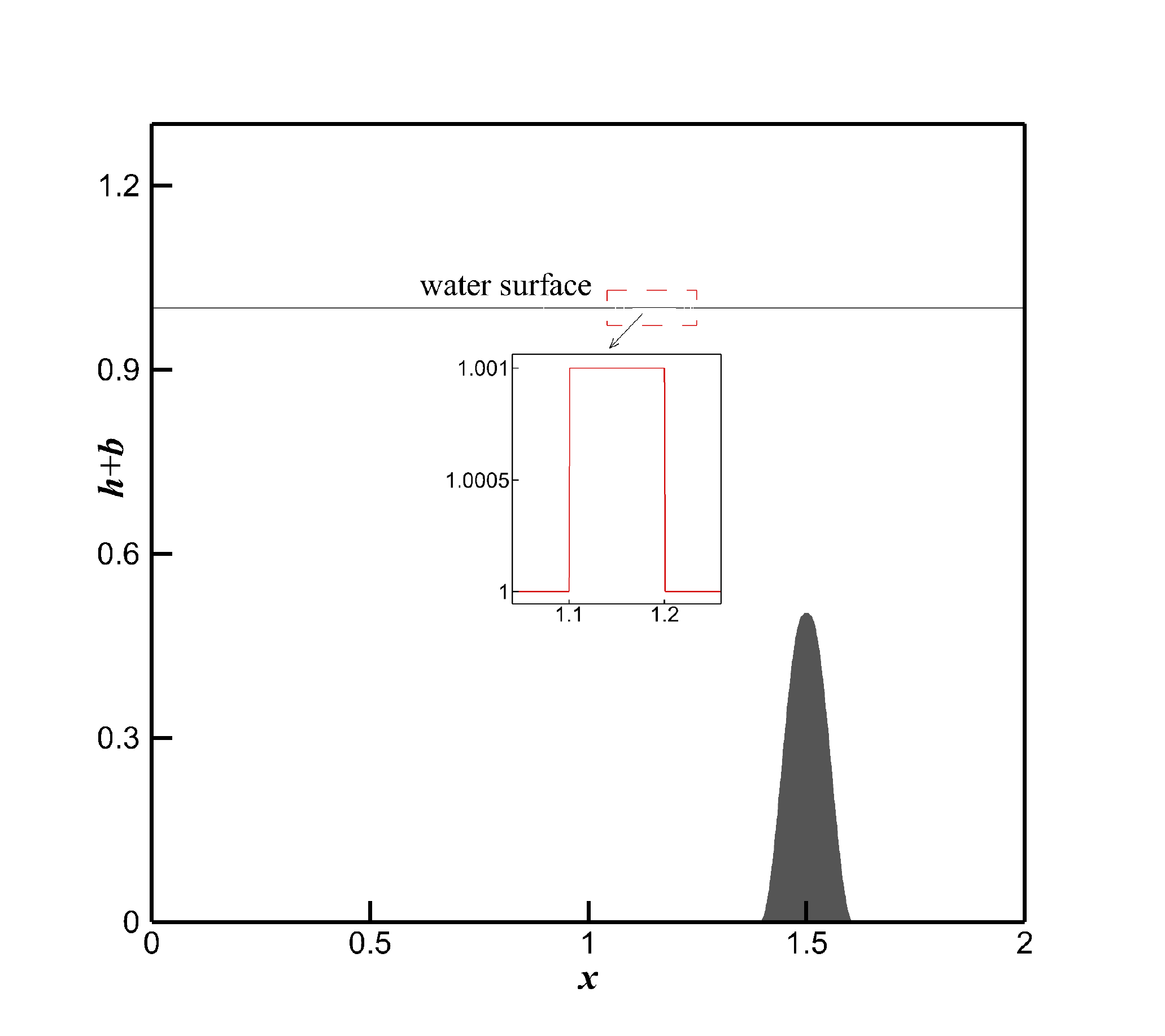}}
\caption{Example \ref{test3-1d}. The initial free water surface level $h+b$ and bottom topography $b$ for the small perturbations $\varepsilon=0.2$ and $\varepsilon=10^{-3}$.}
\label{Fig:test3-1d-initial}
\end{figure}

\begin{figure}[H]
\centering
\subfigure[$h+b$]{
\includegraphics[width=0.45\textwidth,trim=0 0 20 10,clip]{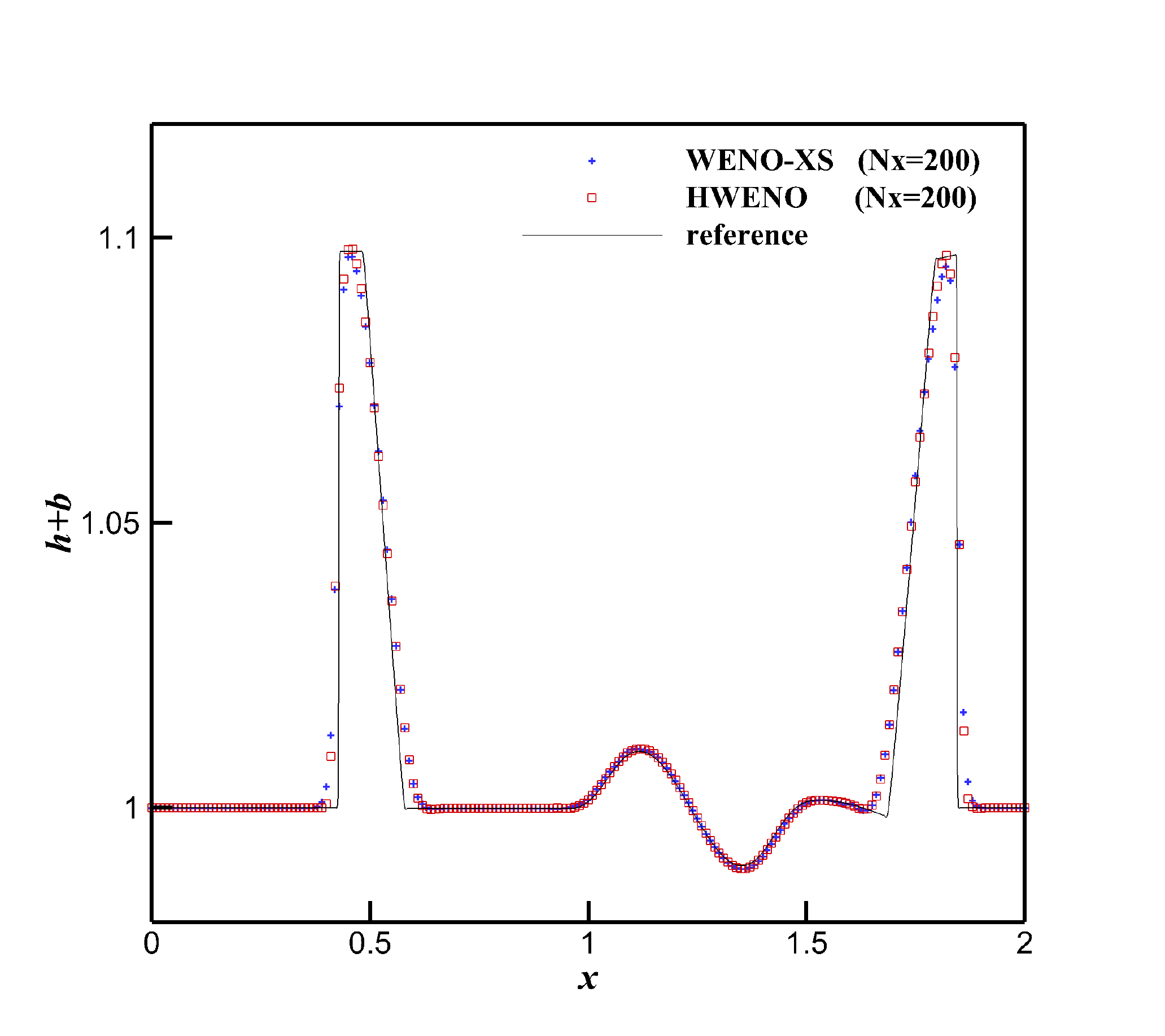}}
\subfigure[$hu$]{
\includegraphics[width=0.45\textwidth,trim=0 0 20 10,clip]{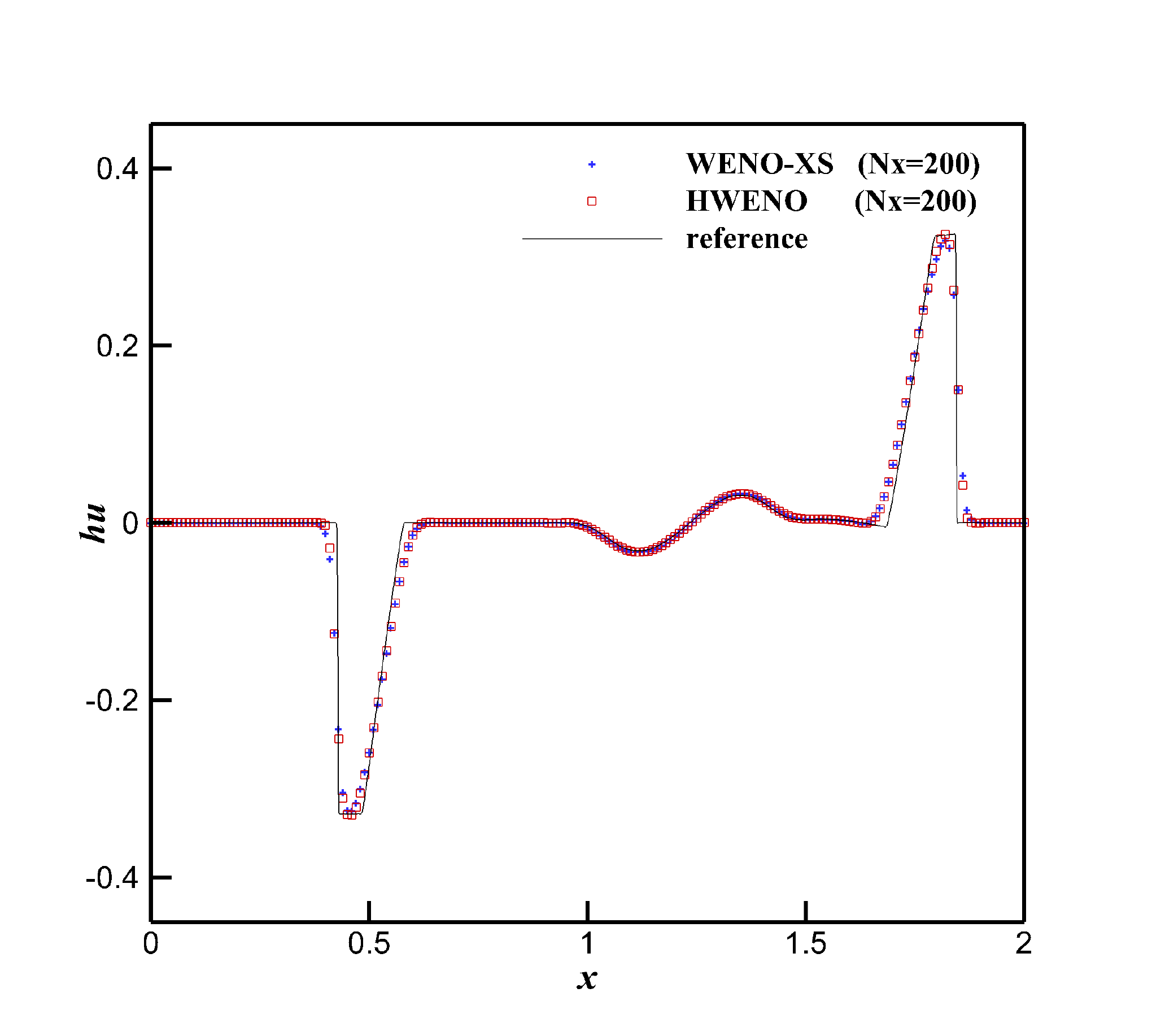}}
\caption{Example \ref{test3-1d}. The free water surface $h+b$ and discharge $hu$ at $t=0.2$ for large pulse $\varepsilon =0.2$.}
\label{Fig:test3-1d-case1}
\end{figure}

\begin{figure}[H]
\centering
\subfigure[$h+b$]{
\includegraphics[width=0.45\textwidth,trim=0 0 20 10,clip]{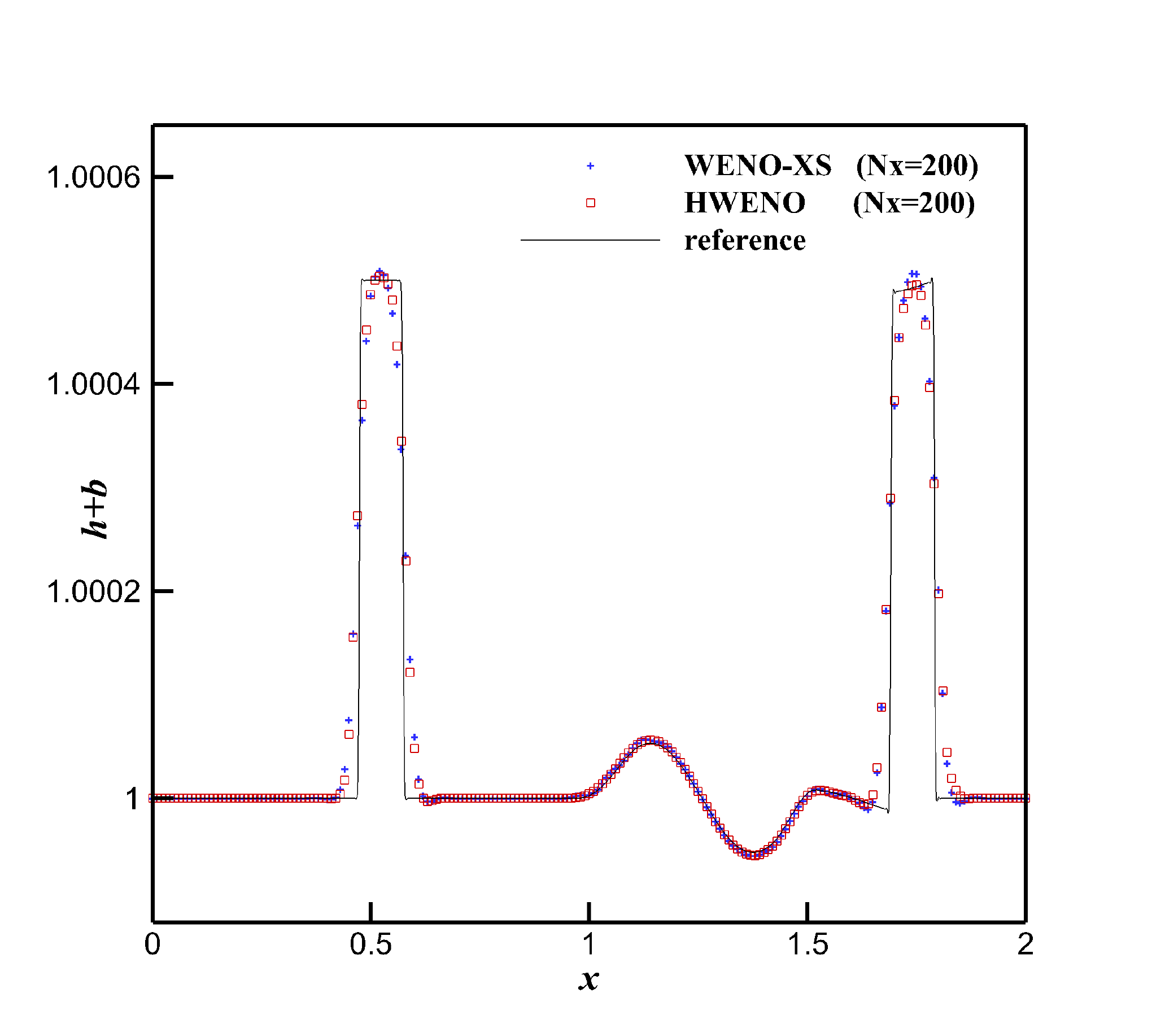}}
\subfigure[$hu$]{
\includegraphics[width=0.45\textwidth,trim=0 0 20 10,clip]{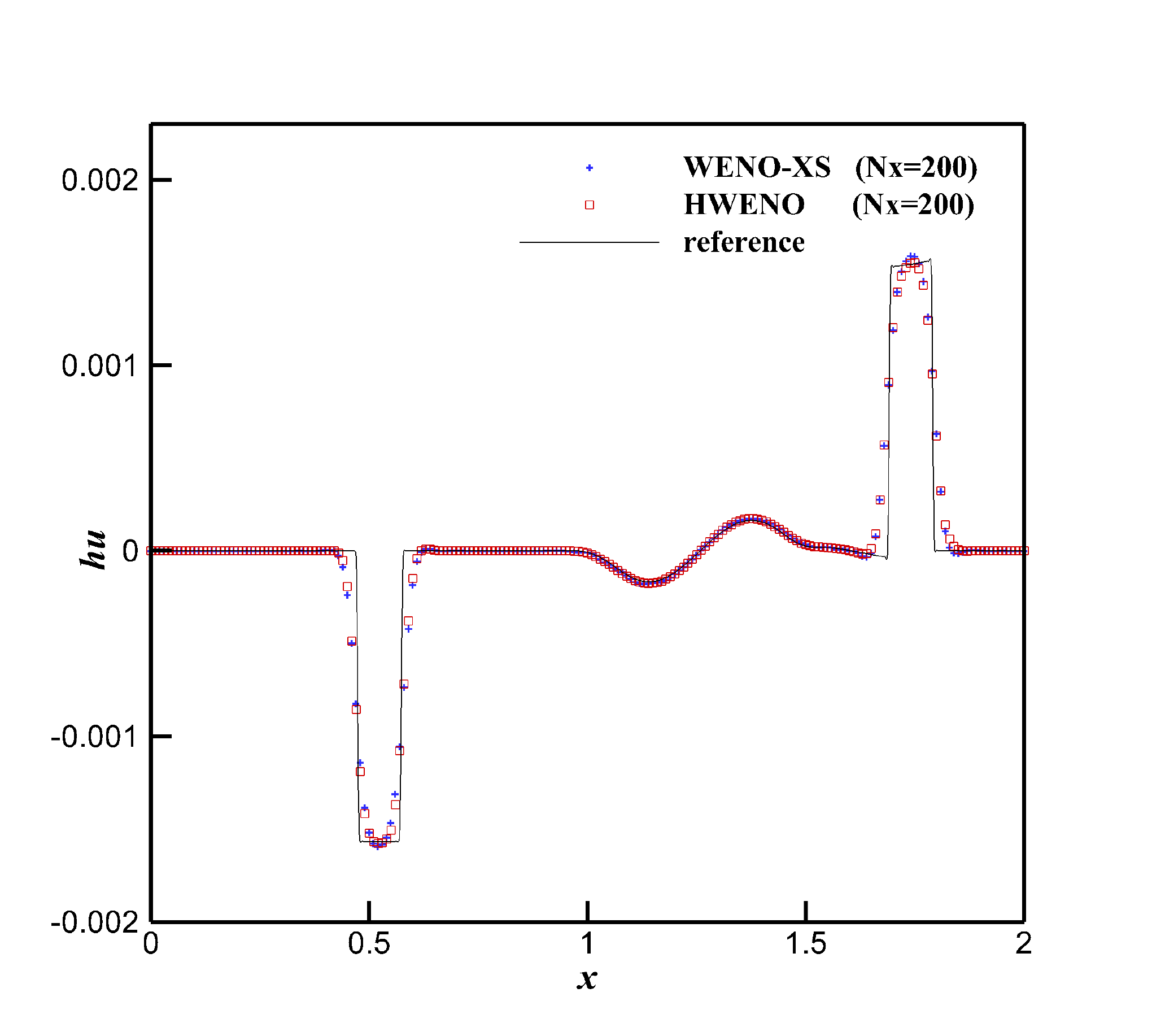}}
\caption{Example \ref{test3-1d}. The free water surface $h+b$ and discharge $hu$ at $t=0.2$ for small pulse $\varepsilon =10^{-3}$.}
\label{Fig:test3-1d-case2}
\end{figure}

\begin{example}\label{test6-1d}
(The rarefaction and shock waves test for the 1D SWEs with wavy bottom topography.)
\end{example}
In this test we compute the 1D SWEs with a wavy bottom topography \cite{HzTang-2004} as
\begin{equation}
\label{B-3}
b(x)=
\begin{cases}
0.3 \cos^{30}(\frac{\pi}{2}(x-1)),& 0\leq x\leq 2,\\
0,& \text{otherwise}.
\end{cases}
\end{equation}
The initial conditions are given by
\begin{equation*}
h(x,0)=
\begin{cases}
2-b(x),& x\in [-10,1],\\
0.35-b(x),& x\in (1, 10], \\
\end{cases}
\quad \quad
u(x,0)=
\begin{cases}
1,& x\in [-10,1],\\
0,& x\in (1, 10].
\end{cases}
\end{equation*}
We compute the solution up to $t=1$.

The solution shows more complex features, made up of a left rarefaction wave and two hydraulic jumps/shocks traveling to the right.
One shock near the location at $x=2$ occurs in the flow over the non-flat bed topography.
In Fig.~\ref{Fig:test6-1d}, we plot the free water surface level $h+b$ and water discharge $hu$ at $t=1$ obtained by the HWENO and WENO-XS schemes with $N_x=200$.
From the figure, we can clearly see that the HWENO scheme works well for this example, giving well resolved, and non-oscillatory solutions, and has better resolutions than WENO-XS scheme in resolving the shock near $x=2$.

\begin{figure}[H]
\centering
\subfigure[$h+b$ ]{
\includegraphics[width=0.45\textwidth,trim=0 0 20 10,clip]{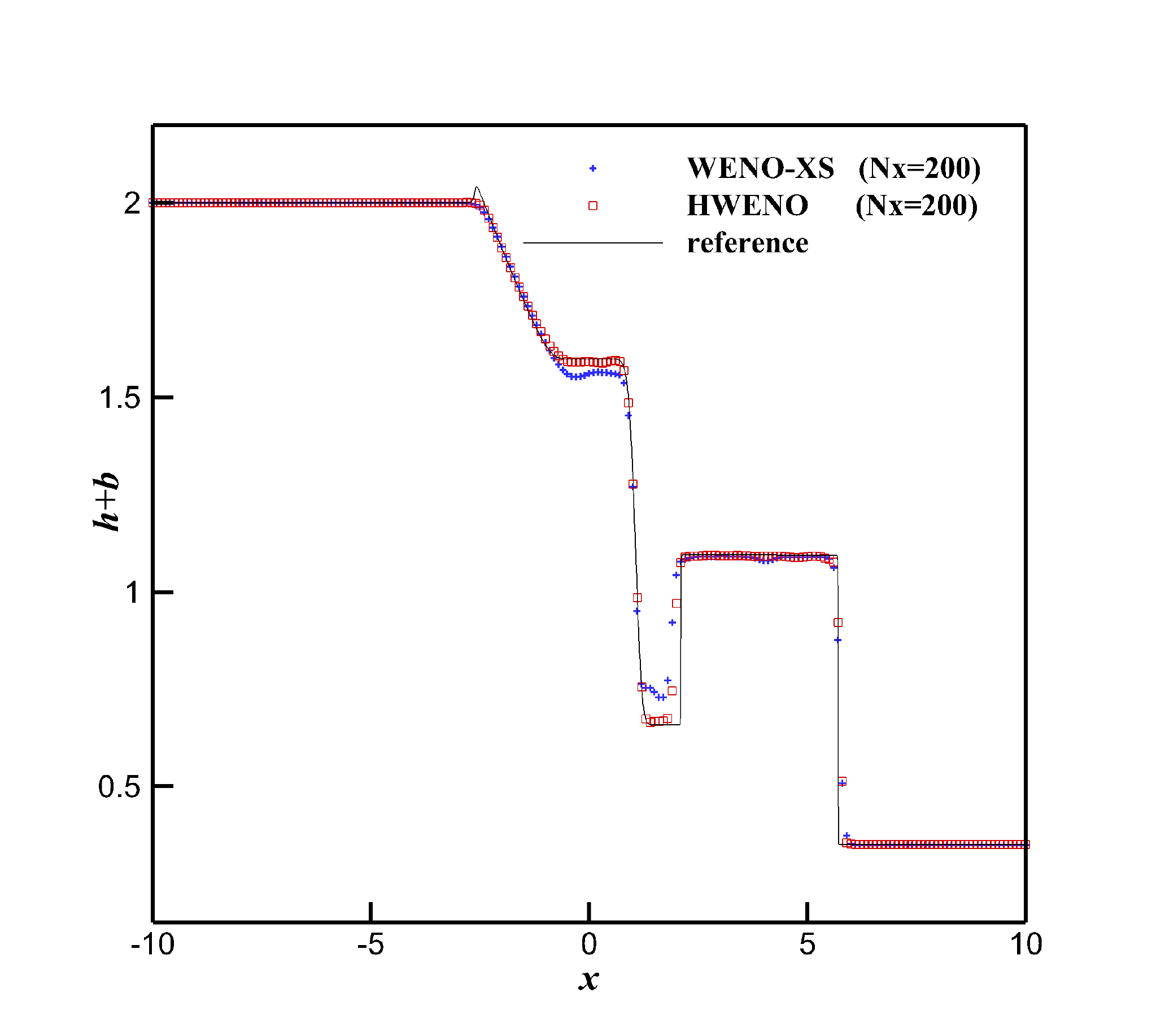}}
\subfigure[$hu$ ]{
\includegraphics[width=0.45\textwidth,trim=0 0 20 10,clip]{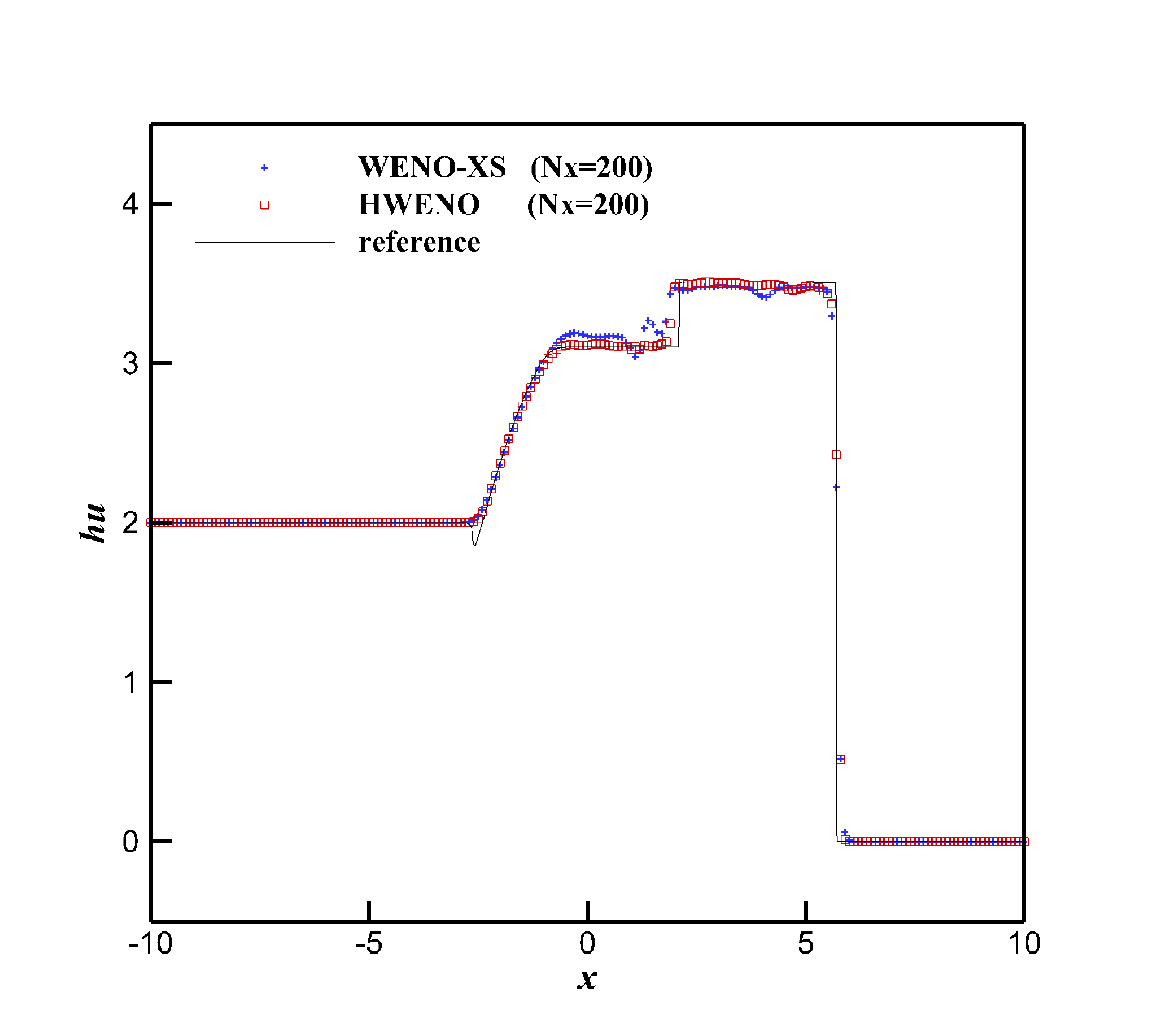}}
\caption{Example \ref{test6-1d}. The free water surface $h+b$ and discharge $hu$ at $t=1$.}
\label{Fig:test6-1d}
\end{figure}

\begin{example}\label{test7-1d}
(The dam break problem for the 1D SWEs over a rectangular bump topography.)
\end{example}
In this test, we simulate the dam break problem over a rectangular bump which produces a rapidly varying flow over a discontinuous bottom topography.
In the computation domain $(0, 1500)$, the bottom topography consists of one rectangular hump
\begin{equation*}
b(x)=
\begin{cases}
8,& |x-750|\leq\frac{1500}{8},\\
0,& \text{otherwise}.
\end{cases}
\end{equation*}
The initial conditions are given by
\begin{equation*}
h(x,0)=
\begin{cases}
20-b(x),& x\leq 750,\\
15-b(x),& \text{otherwise},\\
\end{cases}
\quad \quad
u(x,0)=0.
\end{equation*}

We compute the solution up to $t=15$ and $t=60$.
In this test, the water depth $h$ is discontinuous at the points $x = 562.5$ and $x = 937.5$ and the free water surface level $h+b$ is smooth there.
The free water surface level $h+b$ and water discharge $hu$ at $t=15$ and $t=60$ obtained by the HWENO and WENO-XS schemes with $N_x=200$ are plotted in Fig.~\ref{Fig:test7-1d}.
One can see that the HWENO scheme also works well for this example, giving well resolved, and non-oscillatory solutions, and the results of the HWENO scheme have a slightly better resolution than those of the WENO-XS scheme.

\begin{figure}[H]
\centering
\subfigure[$h+b$: $t=15$]{
\includegraphics[width=0.45\textwidth,trim=0 0 20 10,clip]{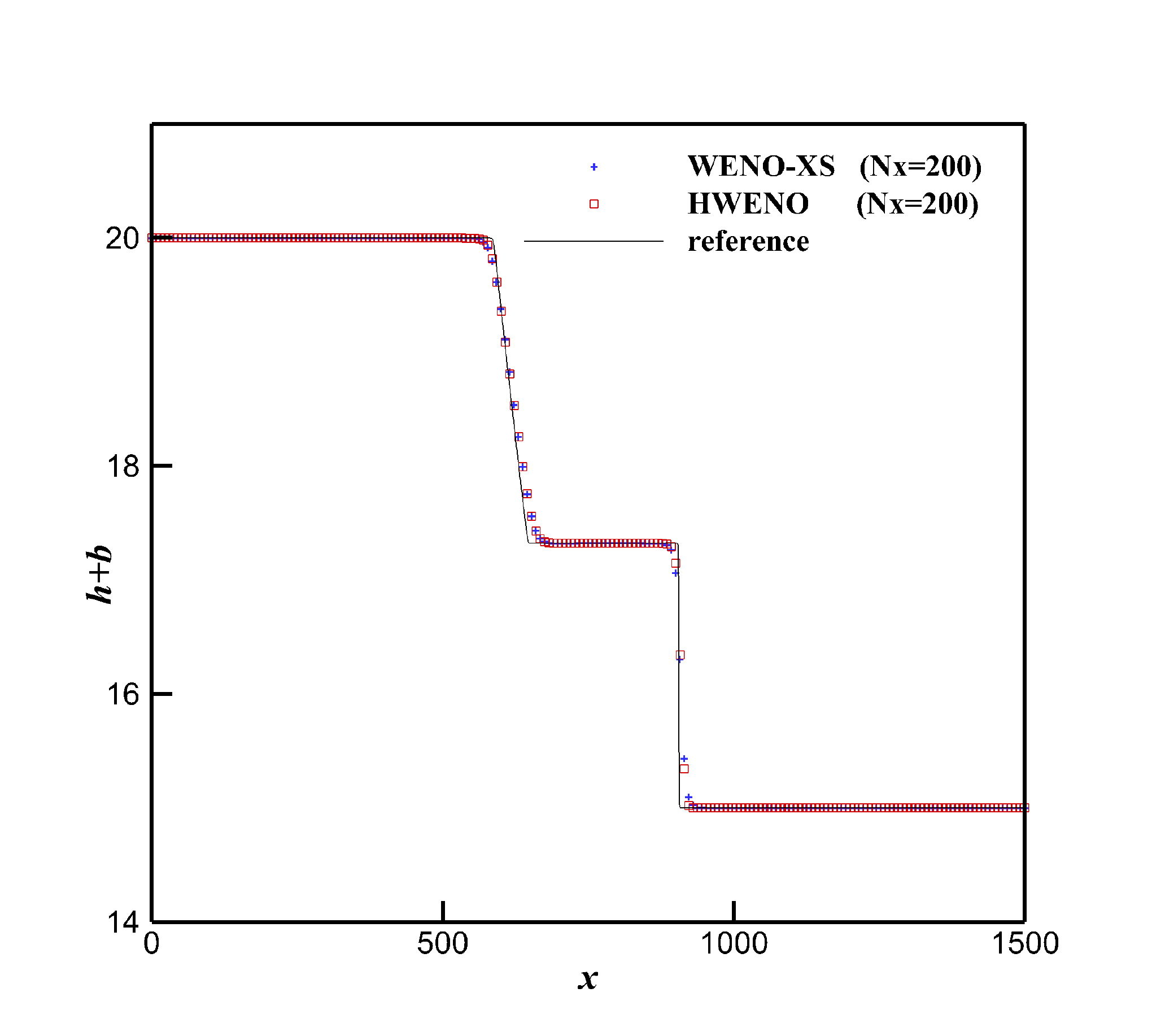}}
\subfigure[$hu$: $t=15$]{
\includegraphics[width=0.45\textwidth,trim=0 0 20 10,clip]{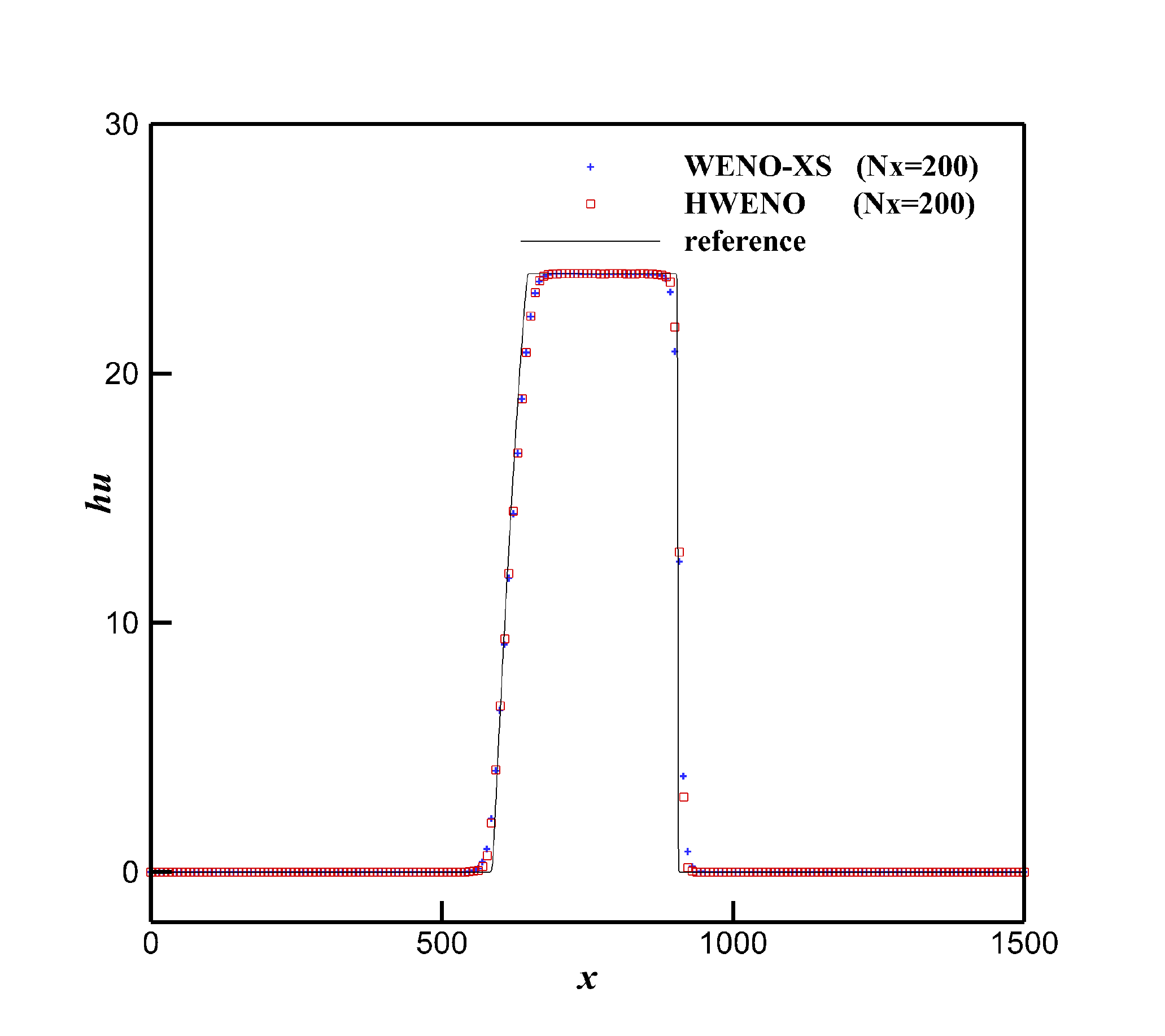}}
\subfigure[$h+b$: $t=60$]{
\includegraphics[width=0.45\textwidth,trim=0 0 20 10,clip]{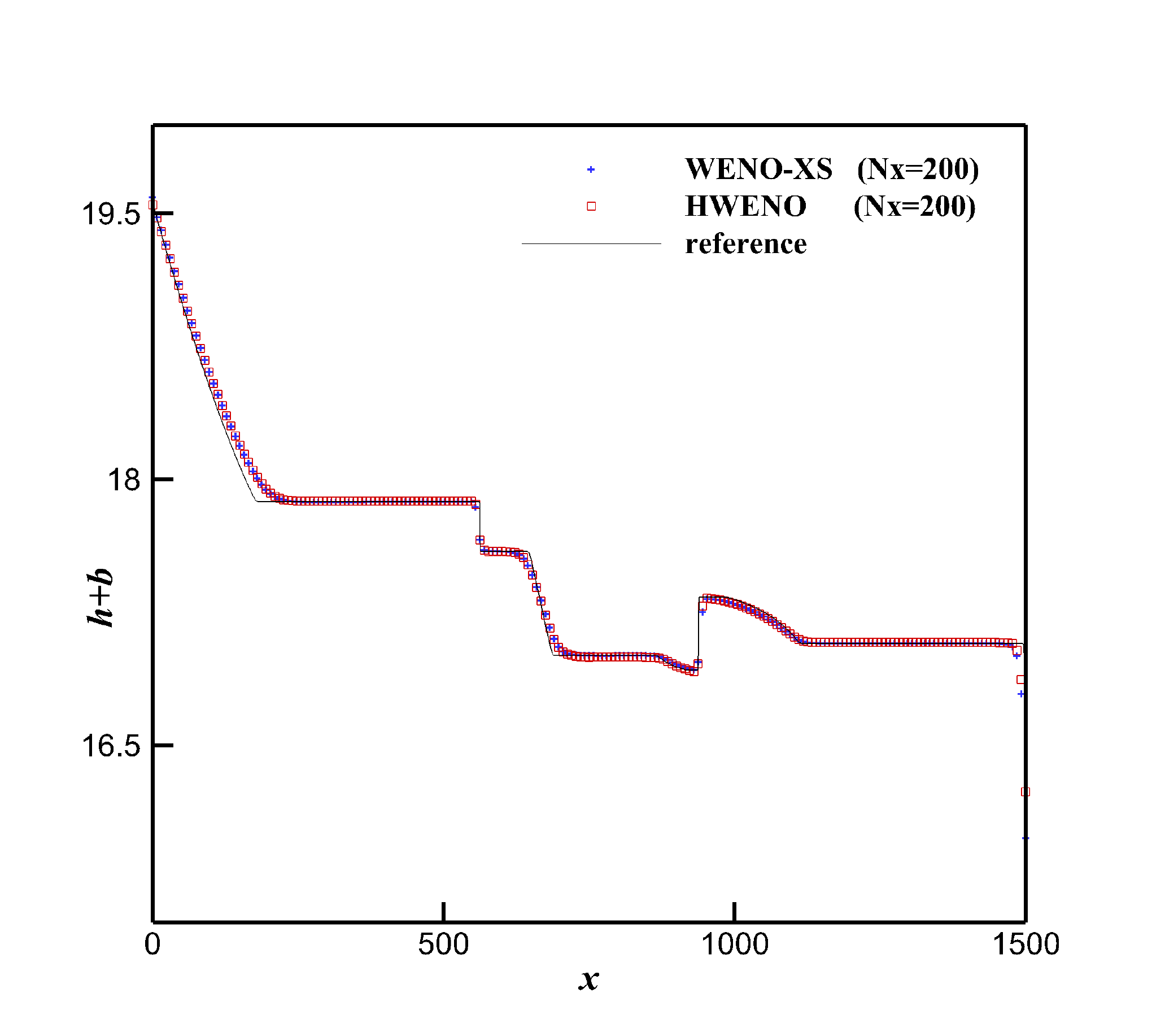}}
\subfigure[$hu$: $t=60$]{
\includegraphics[width=0.45\textwidth,trim=0 0 20 10,clip]{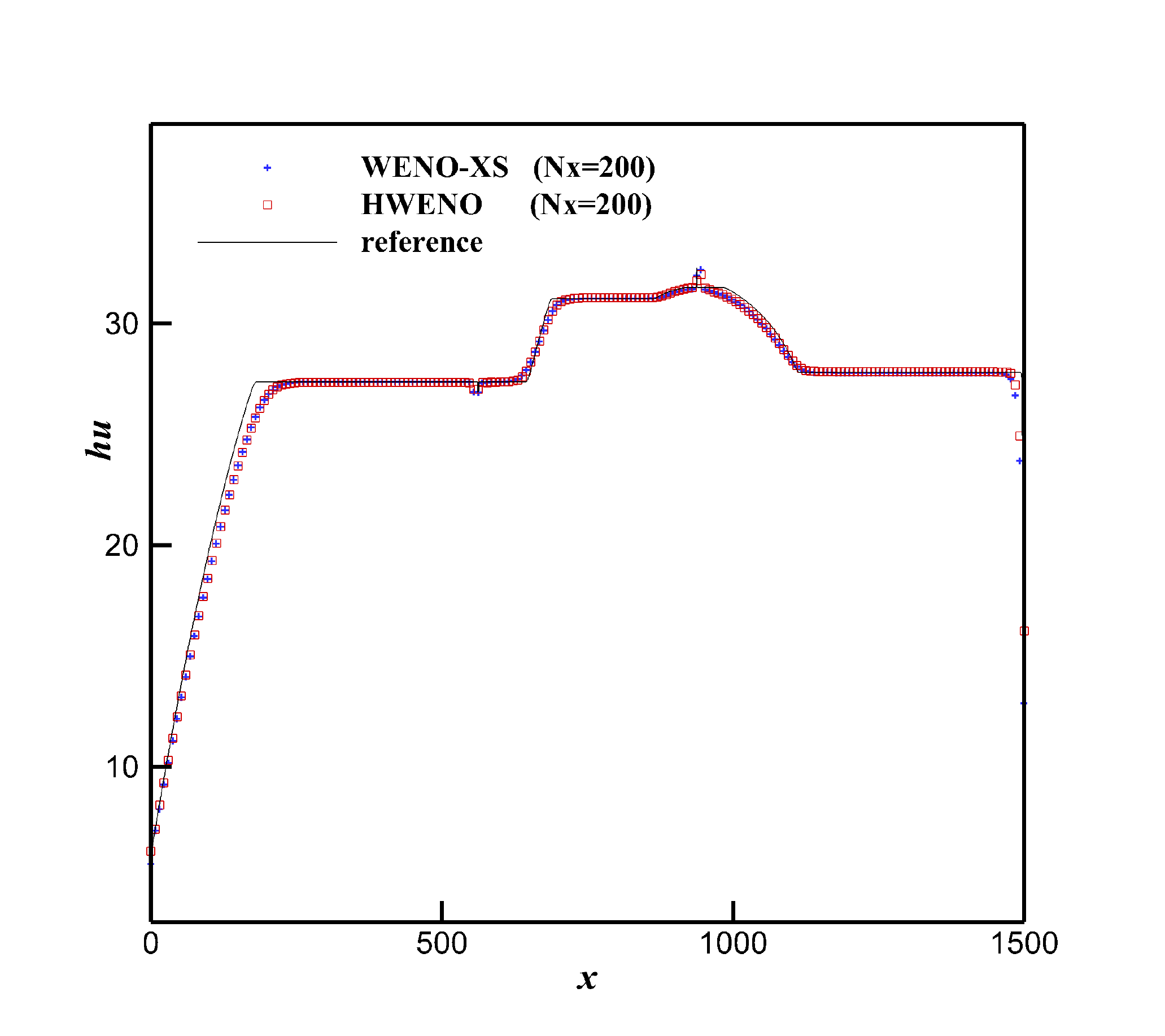}}
\caption{Example \ref{test7-1d}. The free water surface $h+b$ and discharge $hu$ at $t=15$ and $t=60$.}
\label{Fig:test7-1d}
\end{figure}

\begin{example}\label{test1-2d}
(The accuracy test for the 2D SWEs over a sinusoidal hump.)
\end{example}
This example is used to verify the fifth-order accuracy of the proposed HWENO scheme in two dimensions.
The bottom topography is
\begin{equation*}
b(x,y)=\sin(2\pi x)+\sin(2\pi y), \quad (x,y) \in [0,1]\times[0,1].
\end{equation*}
We use periodic boundary conditions for all unknown variables. The initial conditions are given as
\begin{align*}
&h(x,y,0)=10+e^{\sin(2\pi x)}\cos(2\pi y), \\
&hu(x,y,0)=\sin\big(\cos(2\pi x)\big) \sin(2\pi y),\\
&hv(x,y,0)=\cos(2\pi x)\cos\big(\sin(2\pi y)\big).
\end{align*}
We compute the solution up to $t=0.05$ when the solution is still smooth.
A reference solution is obtained using the fifth-order finite difference WENO-XS scheme \cite{Xing-Shu-2005JCP} with $N_x\times N_y =1600\times1600$, and treat this reference solution as the exact solution in computing the numerical errors.

The error of $L^1$ and $L^\infty$ norm of the proposed HWENO and WENO-XS schemes for the water depth $h$ and water discharges $hu$ and $hv$ are plotted in Fig.~\ref{Fig:test1-2d}.
One can be seen that the two schemes have the expected fifth-order convergence in both $L^1$ and $L^\infty$ norm. Again, the figures show that the HWENO scheme produces a smaller error than the WENO-XS scheme on the same number of elements.
Similarly as the one-dimensional case, in each direction, the proposed fifth-order HWENO scheme only needs a compact three-point stencil while the fifth-order WENO-XS scheme needs a five-point stencil in the reconstruction.
Thus, the HWENO scheme more accurate and compact than the WENO-XS scheme.

We also present the error of $L^1$ norm against the CPU time for the water depth $h$ and water discharges $hu$ and $hv$ in Fig.~\ref{Fig:test1-2d-cpu},
which illustrates that the HWENO scheme is slightly efficient than the WENO-XS scheme \cite{Xing-Shu-2005JCP} in the sense that the former leads to a smaller error for a fixed amount of the CPU time.

\begin{figure}[H]
\centering
\subfigure[$h$]{
\includegraphics[width=0.31\textwidth,trim=20 10 30 30,clip]{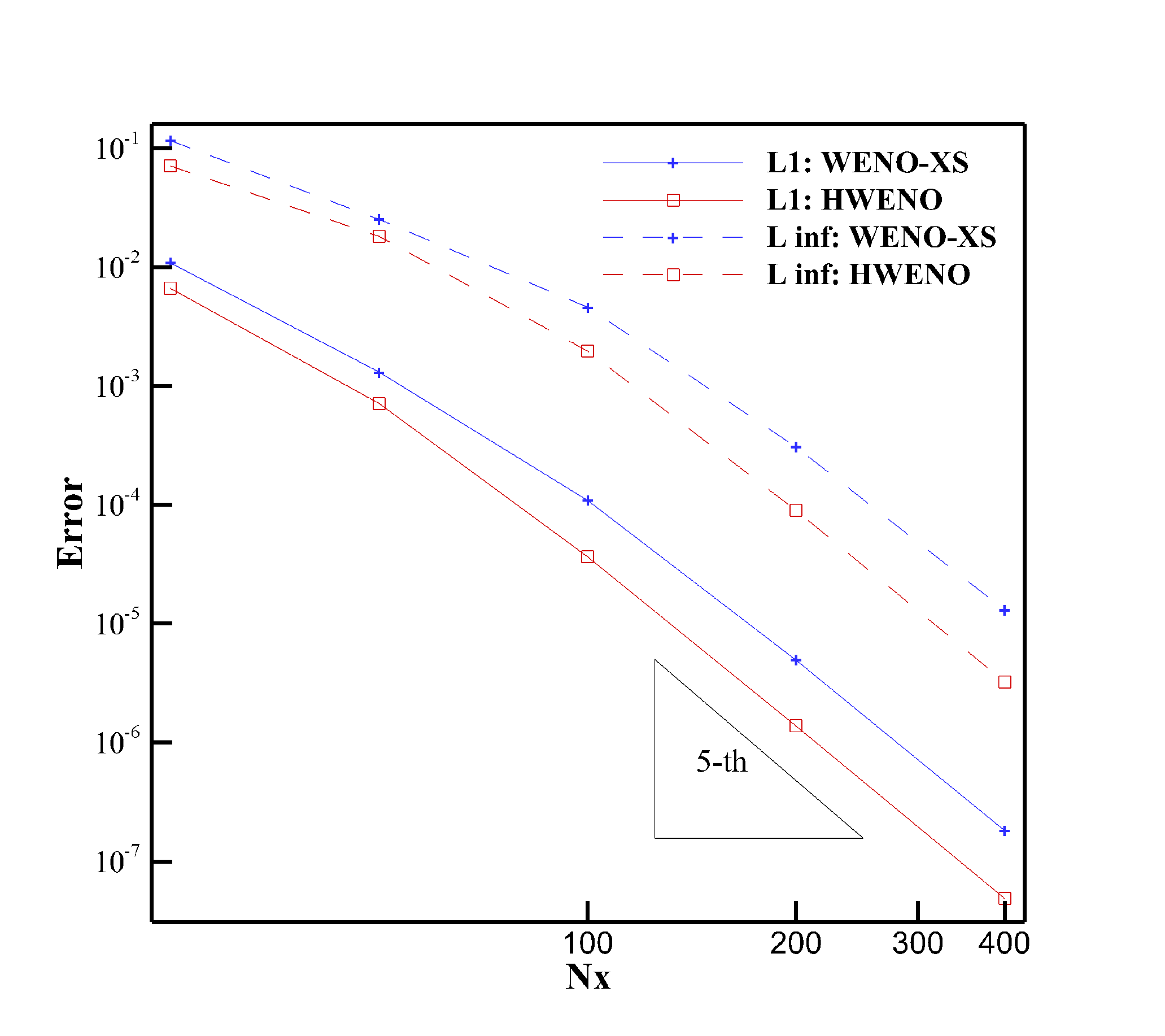}}
\subfigure[$hu$]{
\includegraphics[width=0.31\textwidth,trim=20 10 30 30,clip]{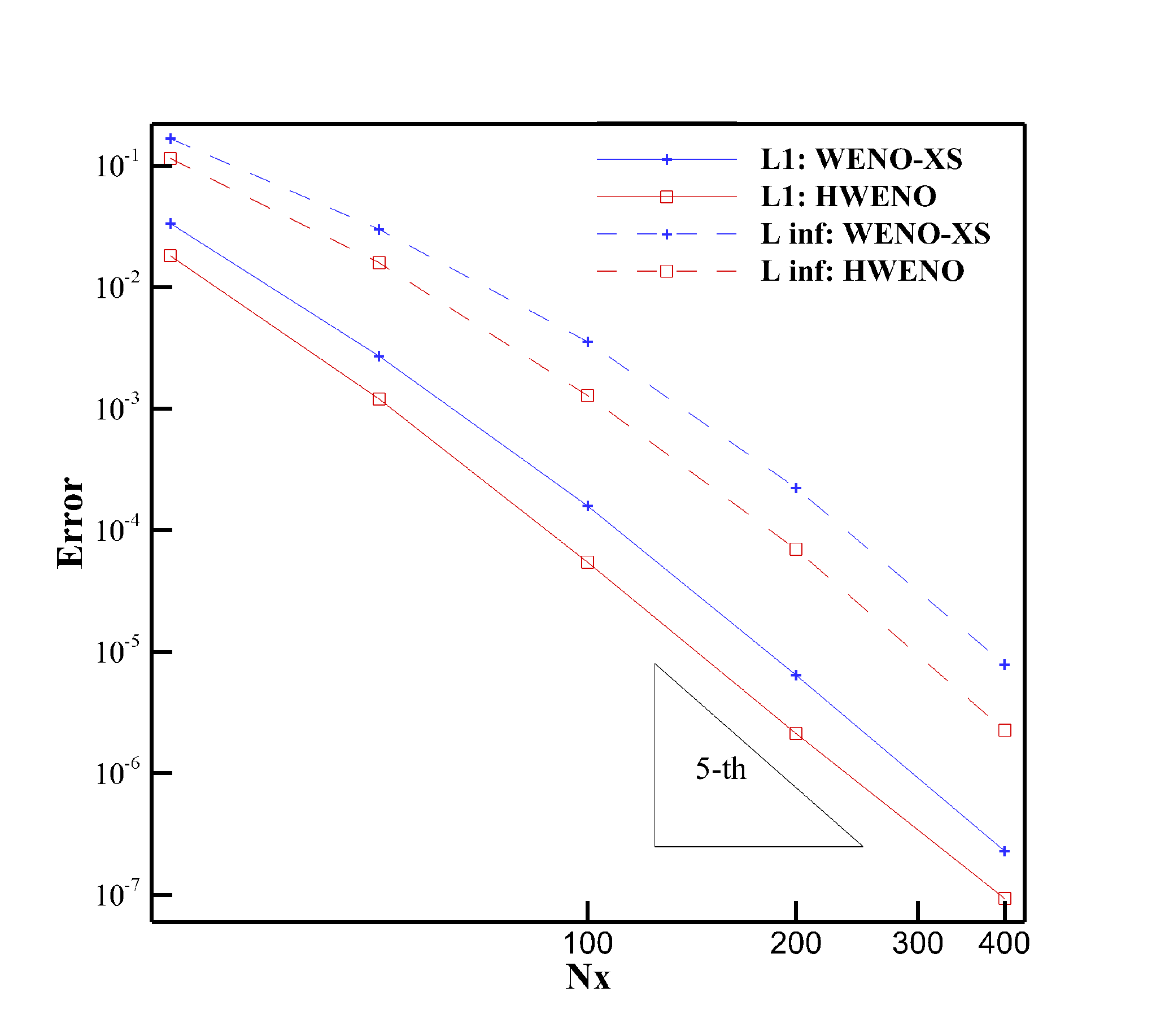}}
\subfigure[$hv$]{
\includegraphics[width=0.31\textwidth,trim=20 10 30 30,clip]{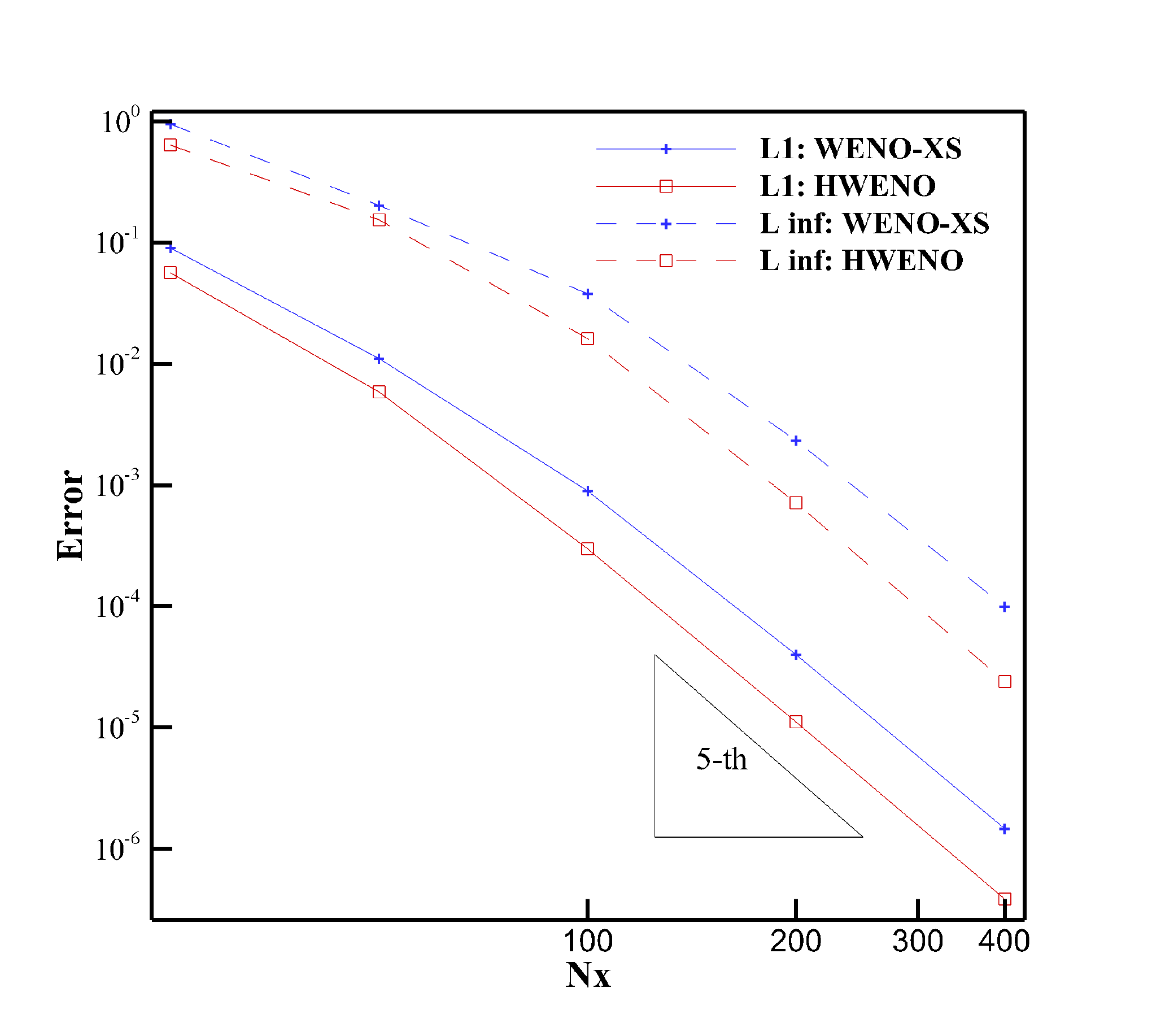}}
\caption{Example \ref{test1-2d}. The error of $L^1$ and $L^\infty$ norm for the water depth $h$ and water discharges $hu$ and $hv$.}
\label{Fig:test1-2d}
\end{figure}

\begin{figure}[H]
\subfigure[$h$]{
\includegraphics[width=0.30\textwidth,trim=20 10 30 30,clip]{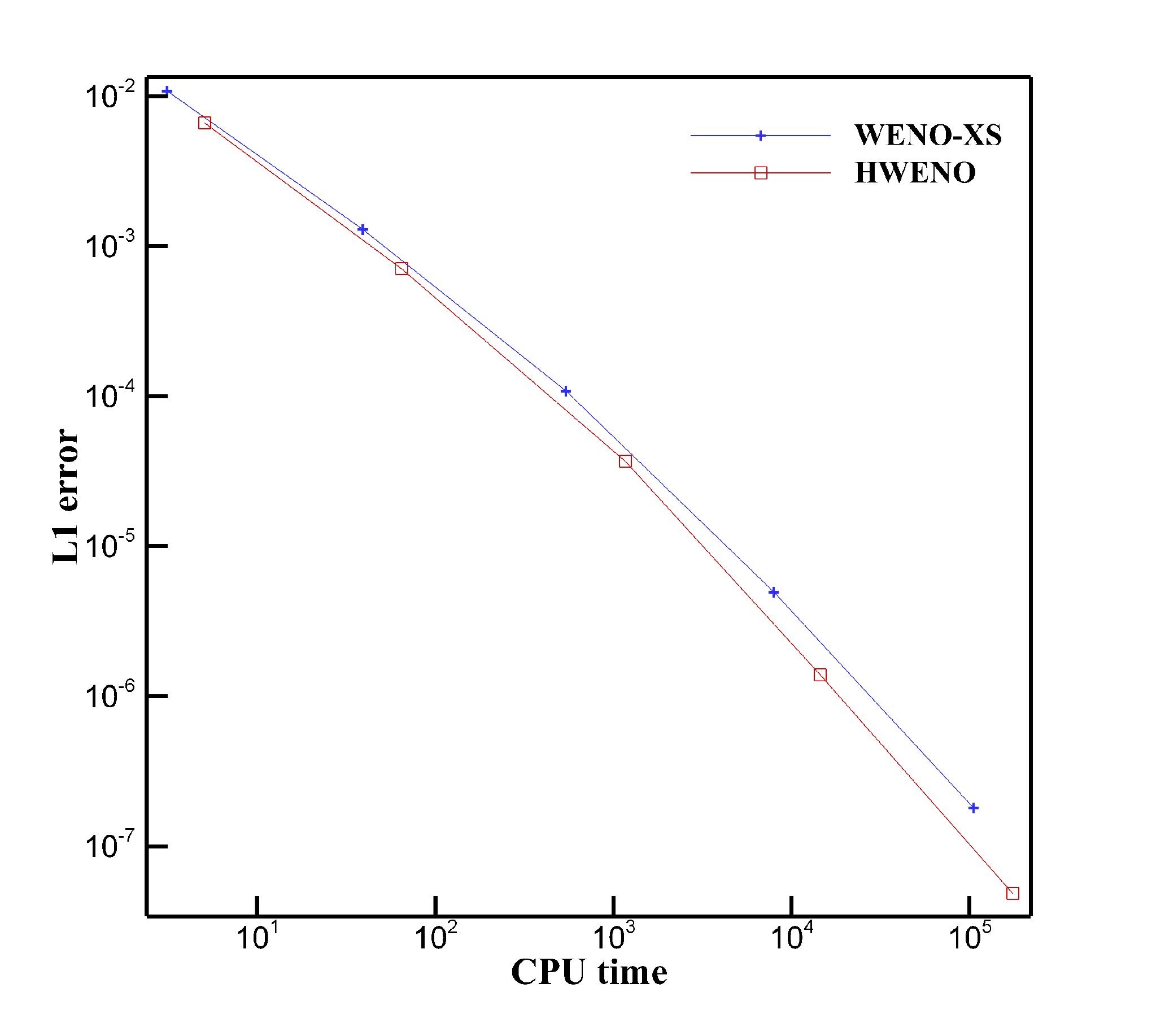}}
\subfigure[$hu$]{
\includegraphics[width=0.30\textwidth,trim=20 10 30 30,clip]{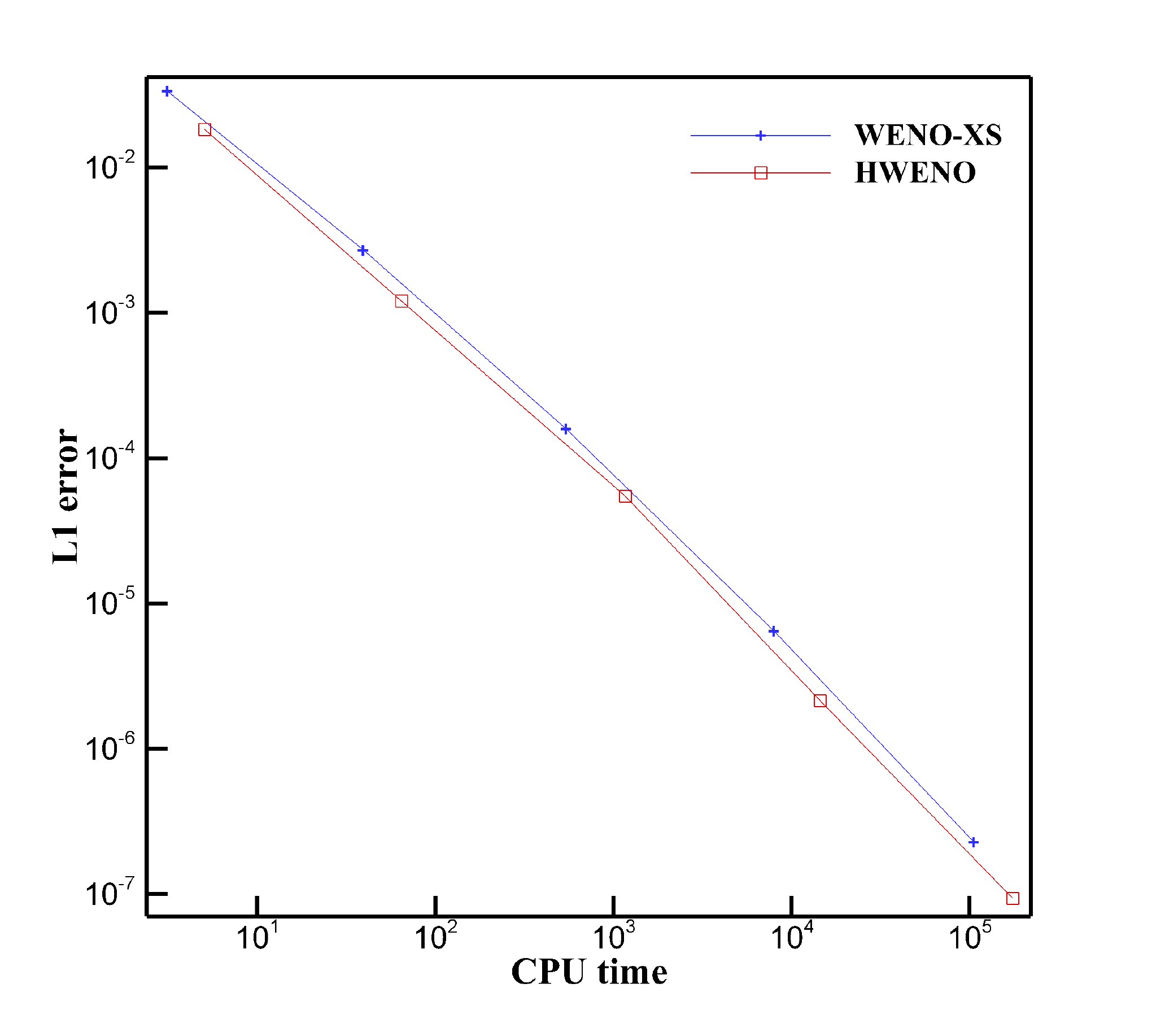}}
\subfigure[$hv$]{
\includegraphics[width=0.30\textwidth,trim=20 10 30 30,clip]{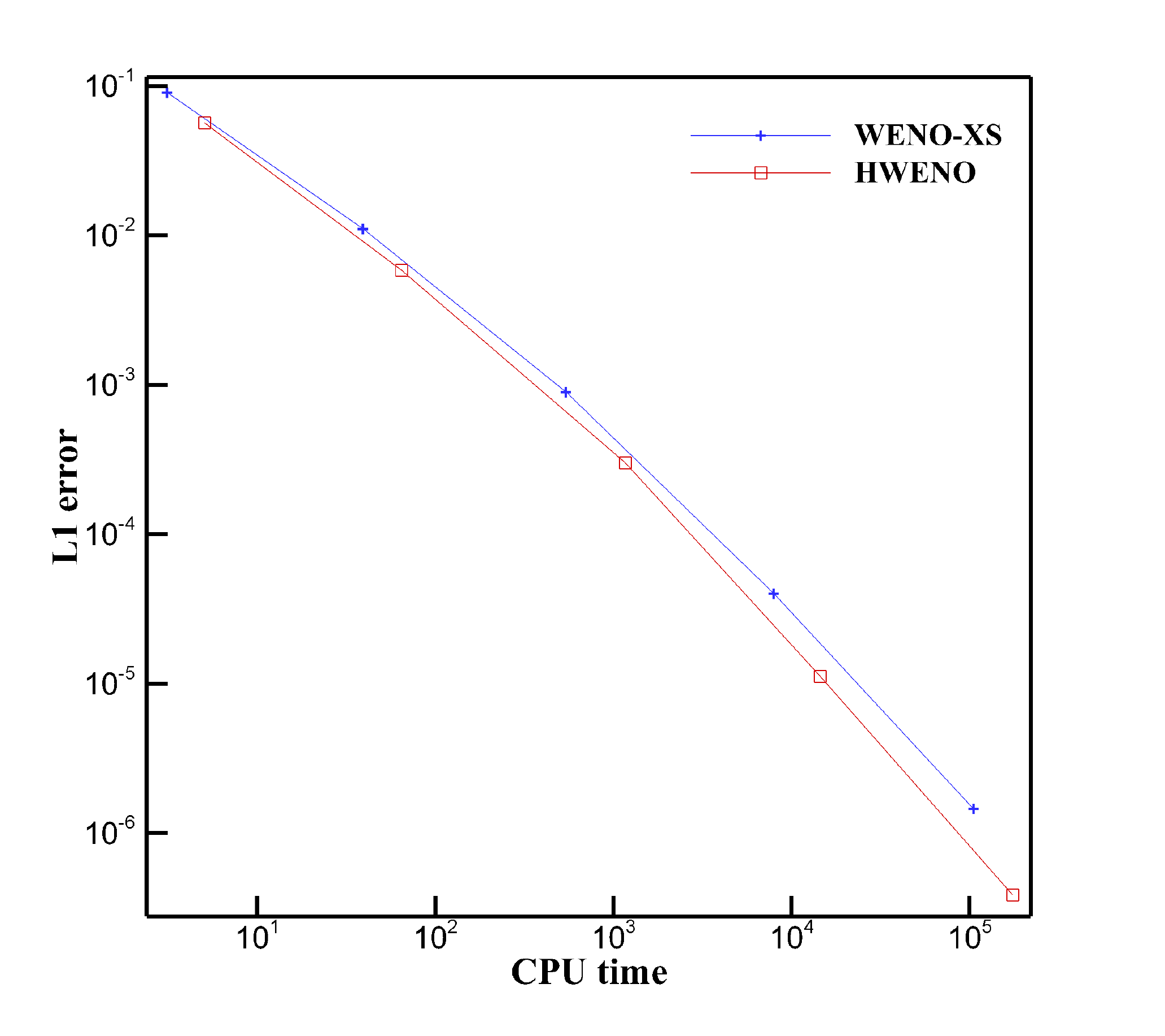}}
 \caption{Example \ref{test1-2d}. The error of $L^1$ norm against the CPU time.}
\label{Fig:test1-2d-cpu}
\end{figure}

\begin{example}\label{test2-2d}
(The lake-at-rest steady-state flow test for the 2D SWEs.)
\end{example}
We choose this example to verify the well-balance property of the proposed HWENO scheme over non-flat bottom topography in two dimensions. We take a bottom topography with an isolated elliptical-shaped bump as
\begin{equation}
\label{B-4}
b(x,y)=0.8e^{-50\big((x-0.5)^2+(y-0.5)^2\big)}, \quad (x, y) \in [0,1]\times[0,1].
\end{equation}
The initial depth of water and velocities are given by
\begin{equation*}
h(x,y,0)=1-b(x,y),
\quad u(x,y,0)=0,\quad  v(x,y,0)=0,
\end{equation*}
with the periodic boundary conditions. The still water state should be remained if the HWENO scheme is well-balanced.
We compute the solution up to $t=0.1$ using single, double and quadruple precisions with a mesh $200\times 200$ to show that the well-balance property is attained up to the level of round-off error.
The $L^1$ and $L^\infty$ error for $h+b$, $hu$, and $hv$ of the HWENO scheme are listed in Table~\ref{tab:test2-2d-error}.
It shows that the HWENO scheme maintains the lake-at-rest steady-state to the level of round-off error (single, double and quadruple precisions) in both $L^1$ and $L^\infty$ norm. Thus, the HWENO scheme is well-balanced.

\begin{table}[H]\small
\centering
\caption{Example \ref{test2-2d}. Well-balanced test over an isolated elliptical-shaped hump bottom topography \eqref{B-4}.}
\label{tab:test2-2d-error}
\medskip
\begin{tabular} {lllllll}
 \toprule
Precision    & \multicolumn{2}{c}{$h$} &   \multicolumn{2}{c}{$hu$} &  \multicolumn{2}{c}{$hv$}\\
\cline{2-3} \cline{4-5} \cline{6-7}
& $L^1$ error & $L^\infty$ error&$L^1$ error& $L^\infty$ error &$L^1$ error& $L^\infty$ error\\
\midrule
Single &  1.10E-05 &     1.26E-05 &     1.52E-06 &     4.27E-06 &     4.19E-07 &     2.37E-06\\
Double &    1.81E-16 &     1.11E-15 &     6.80E-16 &     3.89E-15 &     6.83E-16 &     3.73E-15\\
Quadruple&    1.50E-34 &     9.63E-34 &     6.23E-34 &     3.91E-33 &     6.18E-34 &     3.37E-33\\
\bottomrule
\end{tabular}
\end{table}

\begin{example}\label{test4-2d}
(The perturbed lake-at-rest steady-state flow test for the 2D SWEs.)
\end{example}
We use this example to demonstrate the ability
of capturing  small perturbations over the lake-at-rest water surface for the proposed well-balanced HWENO scheme.
The bottom topography is an isolated elliptical-shaped hump,
\begin{equation*}
b(x,y)=0.8e^{-5(x-0.9)^2-50(y-0.5)^2}, \quad (x,y) \in [0,2]\times[0,1].
\end{equation*}
The initial depth of water and velocities are given by
\begin{equation*}
\begin{split}
&h(x,y,0)=\begin{cases}
1-b(x,y)+0.01,& \hbox{$x\in[0.05 , 0.15],$}\\
1-b(x,y),&\hbox{otherwise,}\\
\end{cases}
\\&u(x,y,0)=0,\quad \quad v(x,y,0)=0.
\end{split}
\end{equation*}
The initial perturbation splits into two waves propagating left and right
at the characteristic speeds $\pm \sqrt{gh}$.

The contours of the free water surface level $h+b$ at $t=0.12,~0.24,~0.36,~0.48,~0.60$ obtained by the proposed HWENO scheme with the meshes $200\times 100$ and $600\times 300$, respectively, are shown in
Fig.~\ref{Fig:test4-2d}. $30$ uniformly spaced contour lines and the same ranges as the WENO-XS scheme \cite{Xing-Shu-2006JCP}, i.e., at time $t = 0.12$ from $0.999703$ to $1.00629$; at time $t = 0.24$ from $0.994836$ to $1.01604$; at time $t = 0.36$ from $0.988582$ to $1.0117$; at time $t = 0.48$ from $0.990344$ to $1.00497$; and at time $t = 0.60$ from $0.995065$ to $1.0056$.
These results show that the HWENO scheme can resolve the complex small features of the flow over a smooth bed very well as WENO-XS scheme.
\begin{figure}
\centering
\subfigure[$N_x\times N_y=200 \times 100$, $t=0.12$]{
\includegraphics[width=0.45\textwidth,trim=20 10 30 30,clip]{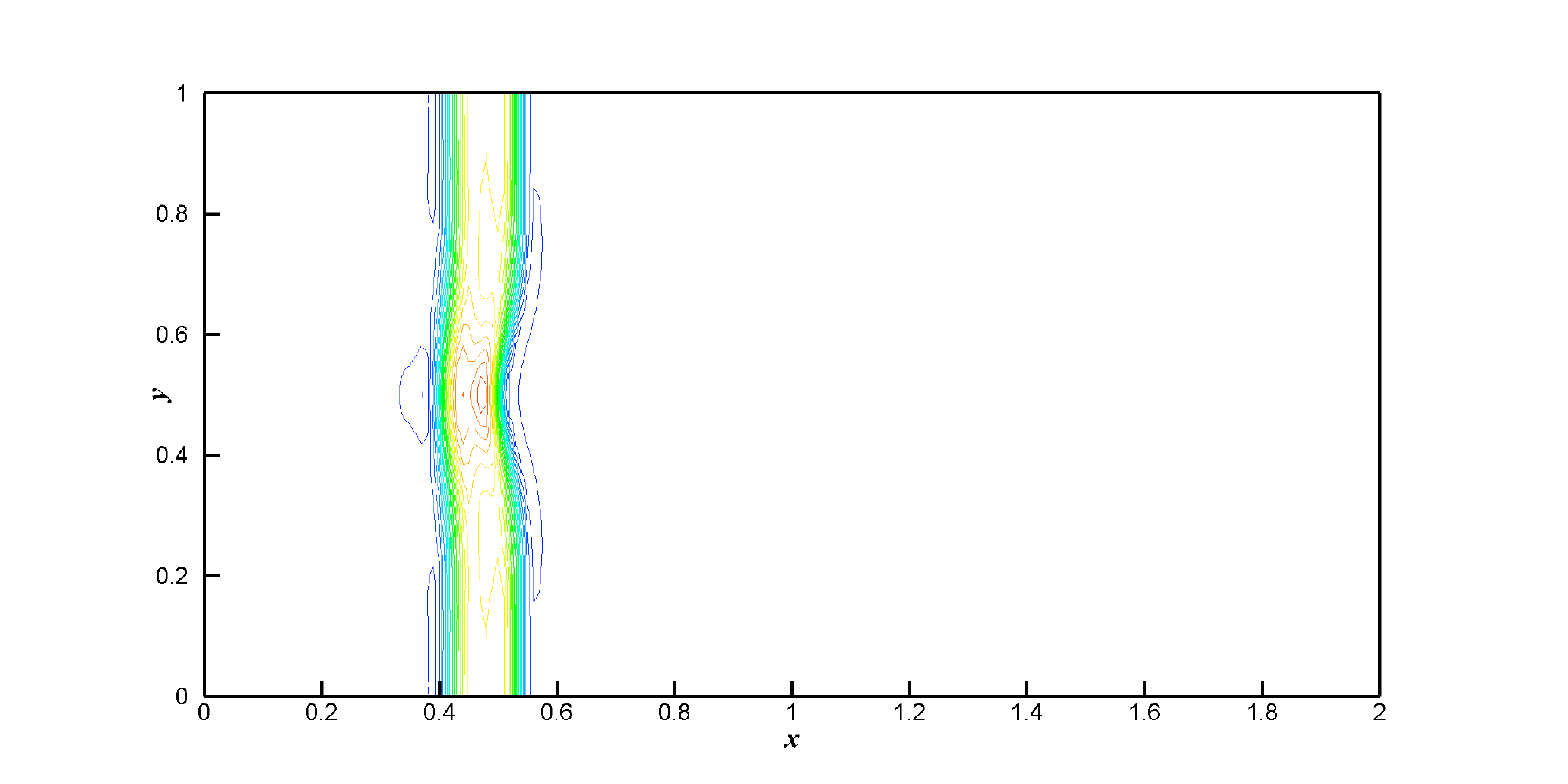}}
\subfigure[$N_x\times N_y=600 \times 300$, $t=0.12$]{
\includegraphics[width=0.45\textwidth,trim=20 10 30 30,clip]{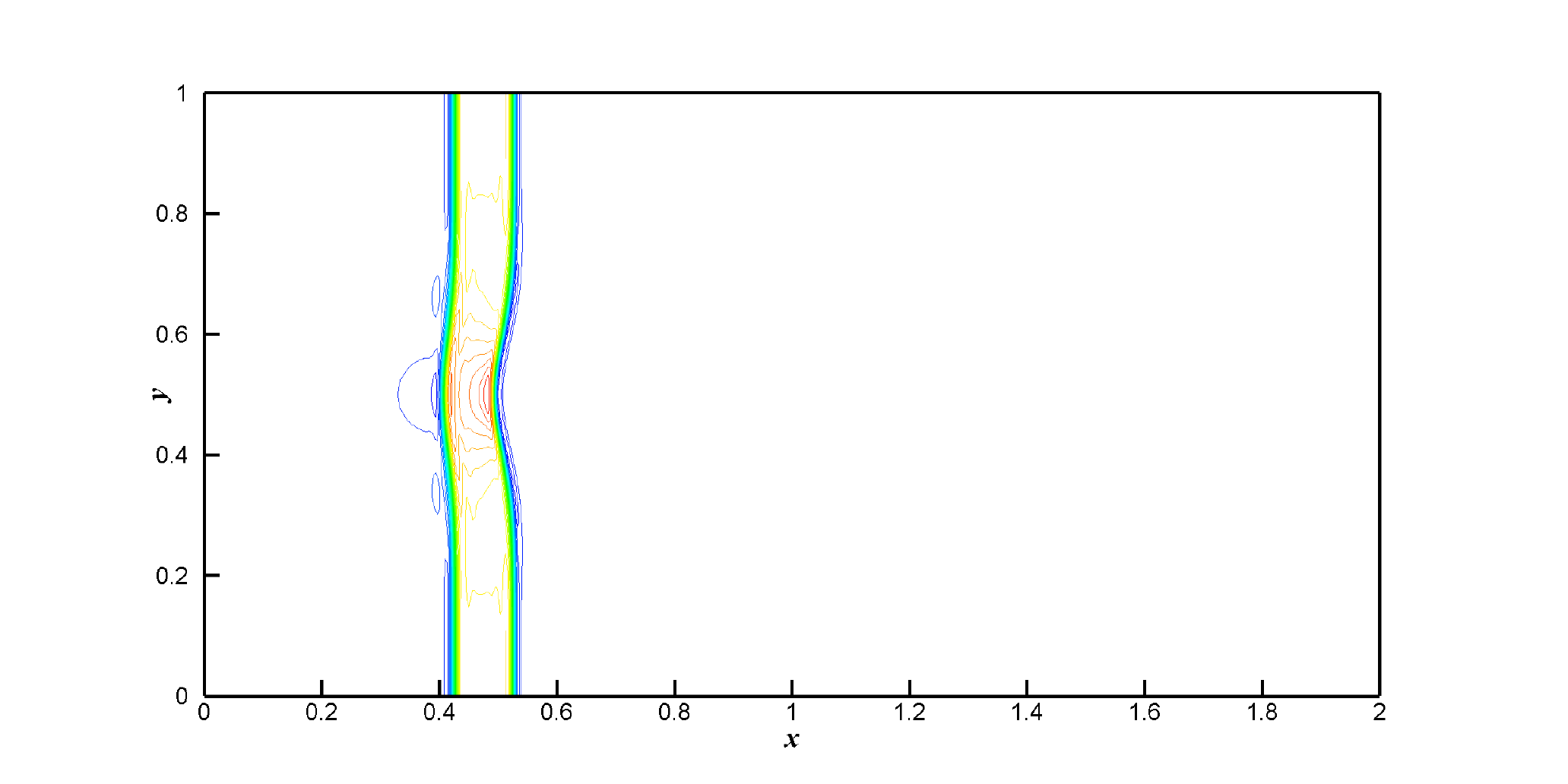}}
\subfigure[$N_x\times N_y=200 \times 100$, $t=0.24$]{
\includegraphics[width=0.45\textwidth,trim=20 10 30 30,clip]{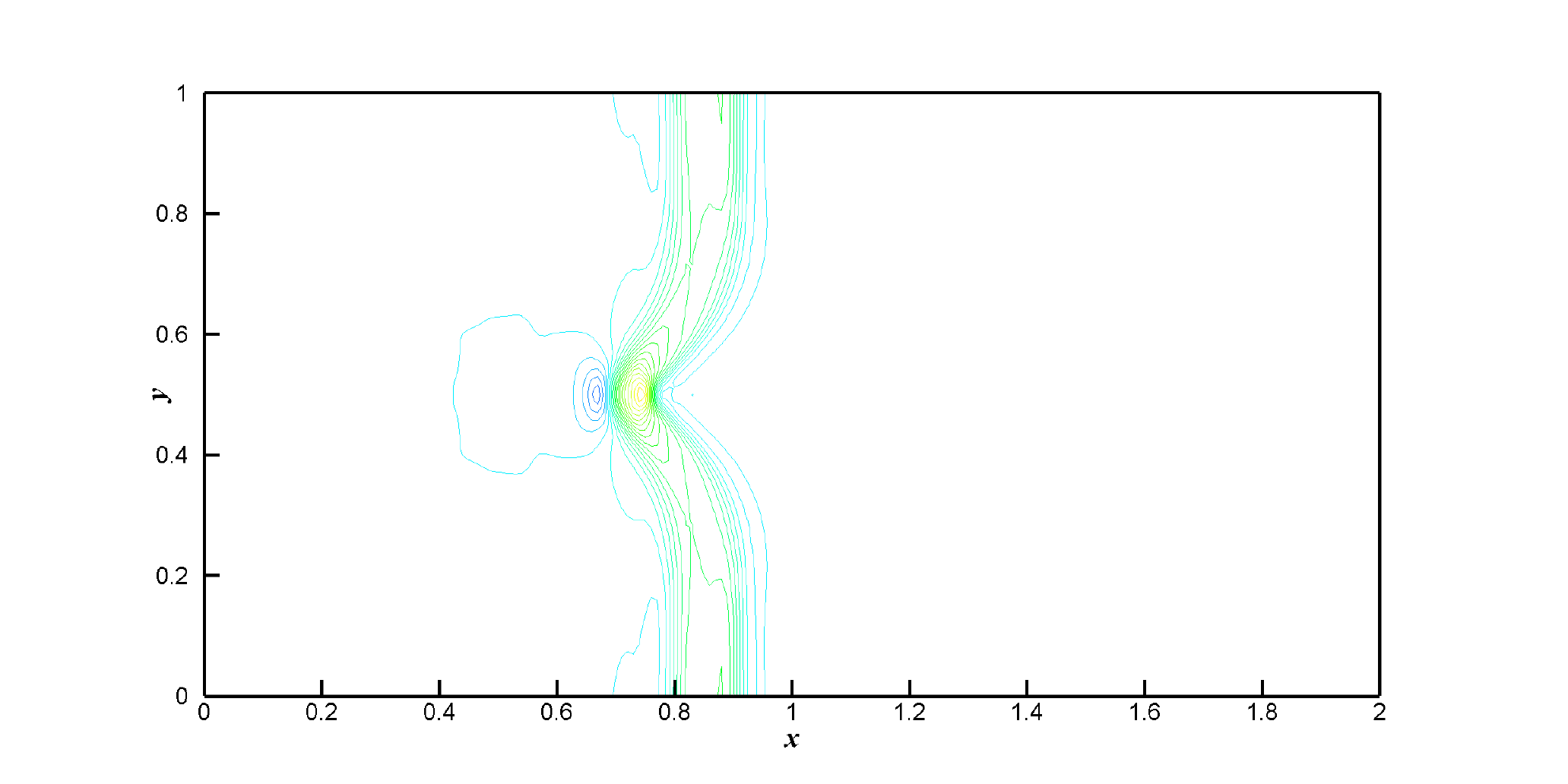}}
\subfigure[$N_x\times N_y=600 \times 300$, $t=0.24$]{
\includegraphics[width=0.45\textwidth,trim=20 10 30 30,clip]{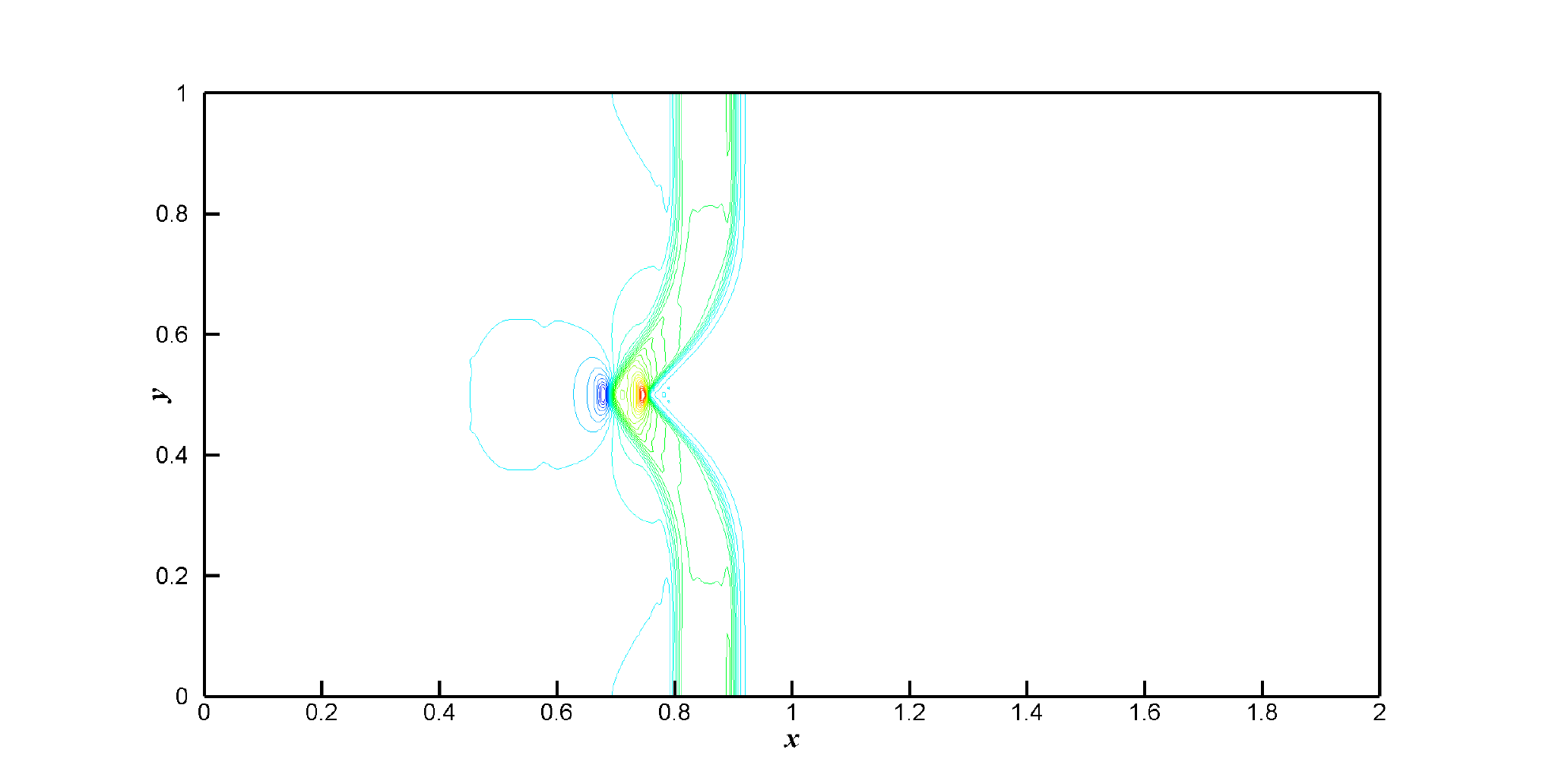}}
\subfigure[$N_x\times N_y=200 \times 100$, $t=0.36$]{
\includegraphics[width=0.45\textwidth,trim=20 10 30 30,clip]{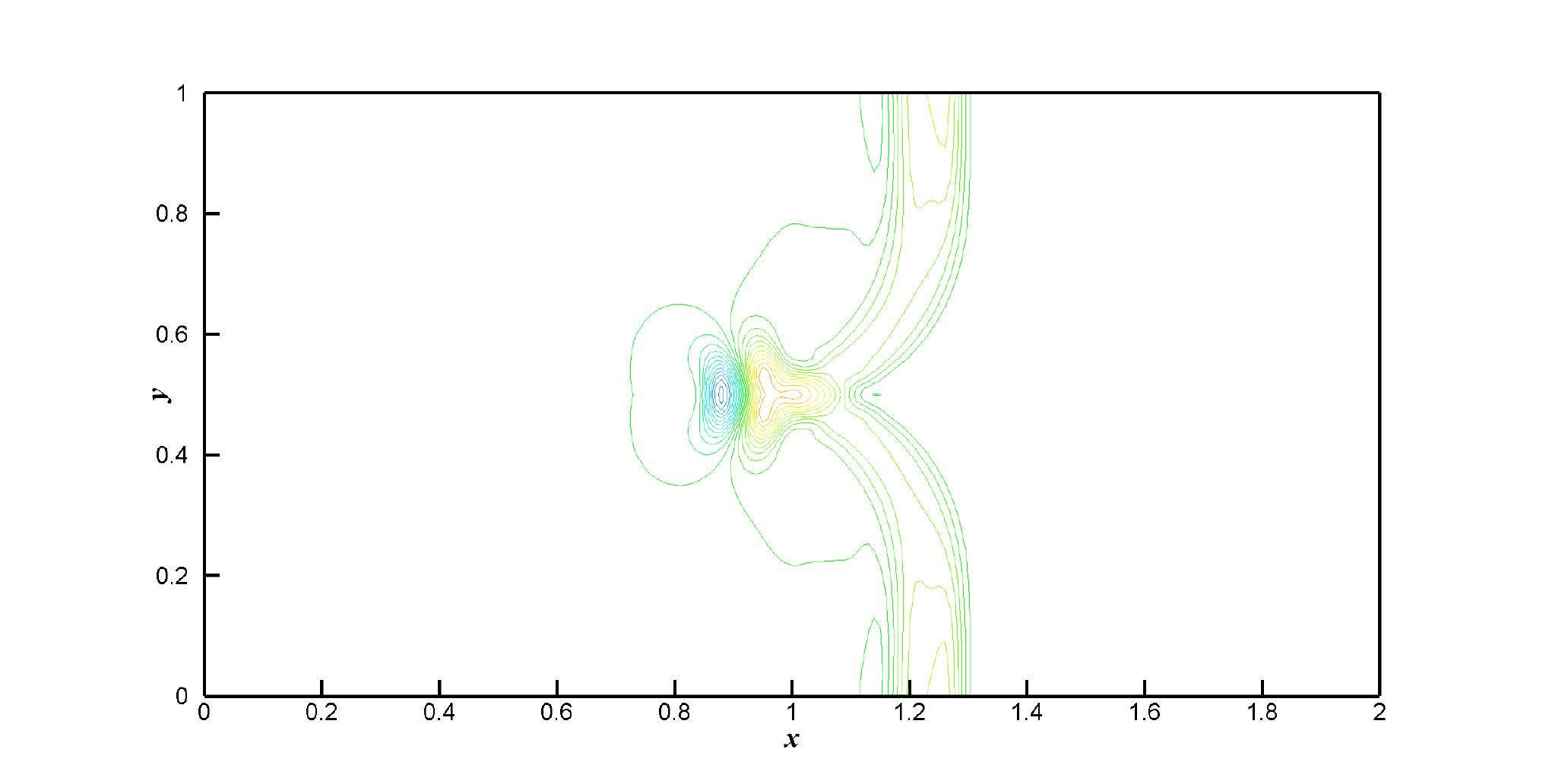}}
\subfigure[$N_x\times N_y=600 \times 300$, $t=0.36$]{
\includegraphics[width=0.45\textwidth,trim=20 10 30 30,clip]{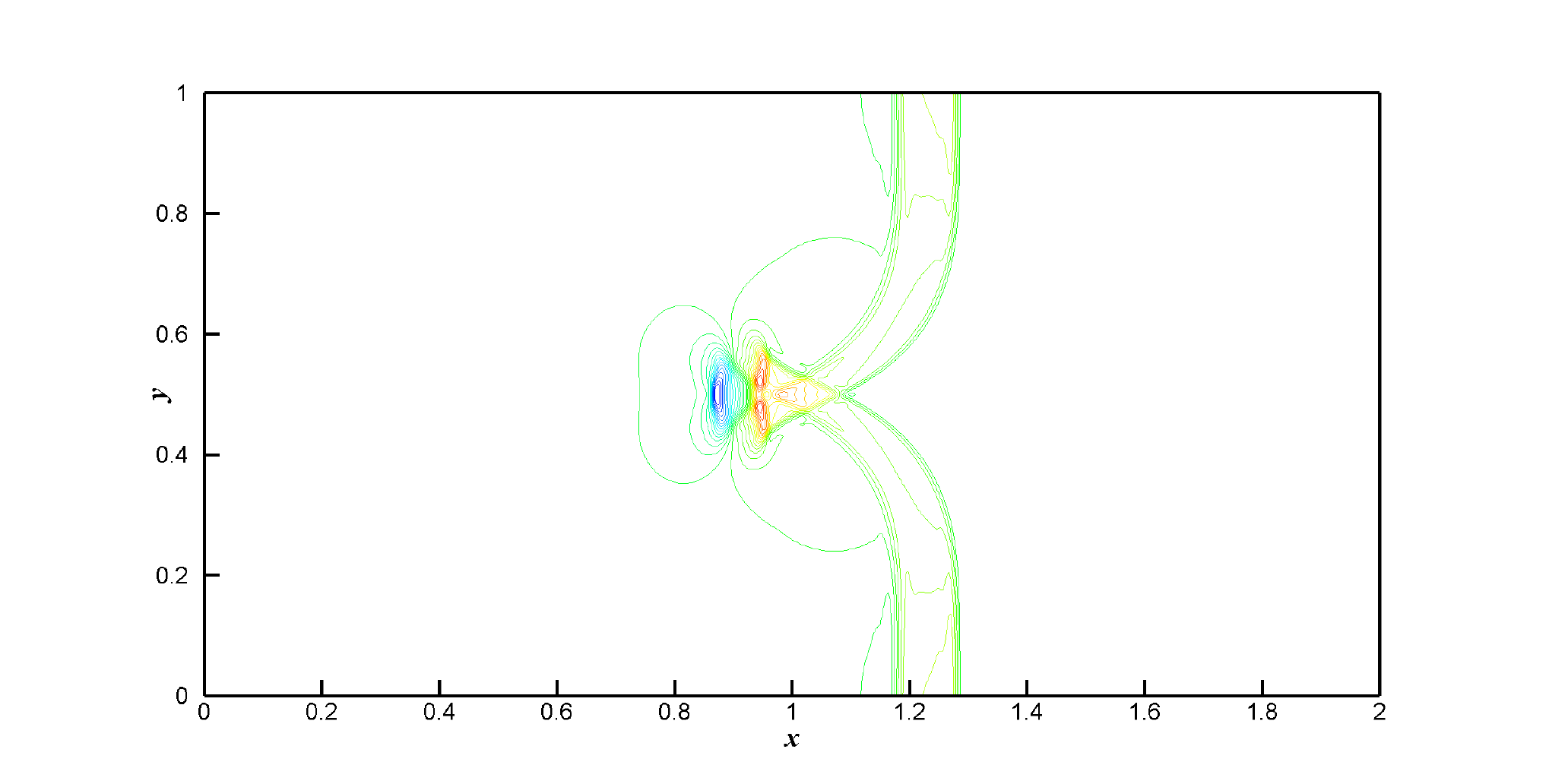}}
\subfigure[$N_x\times N_y=200 \times 100$, $t=0.48$]{
\includegraphics[width=0.45\textwidth,trim=20 10 30 30,clip]{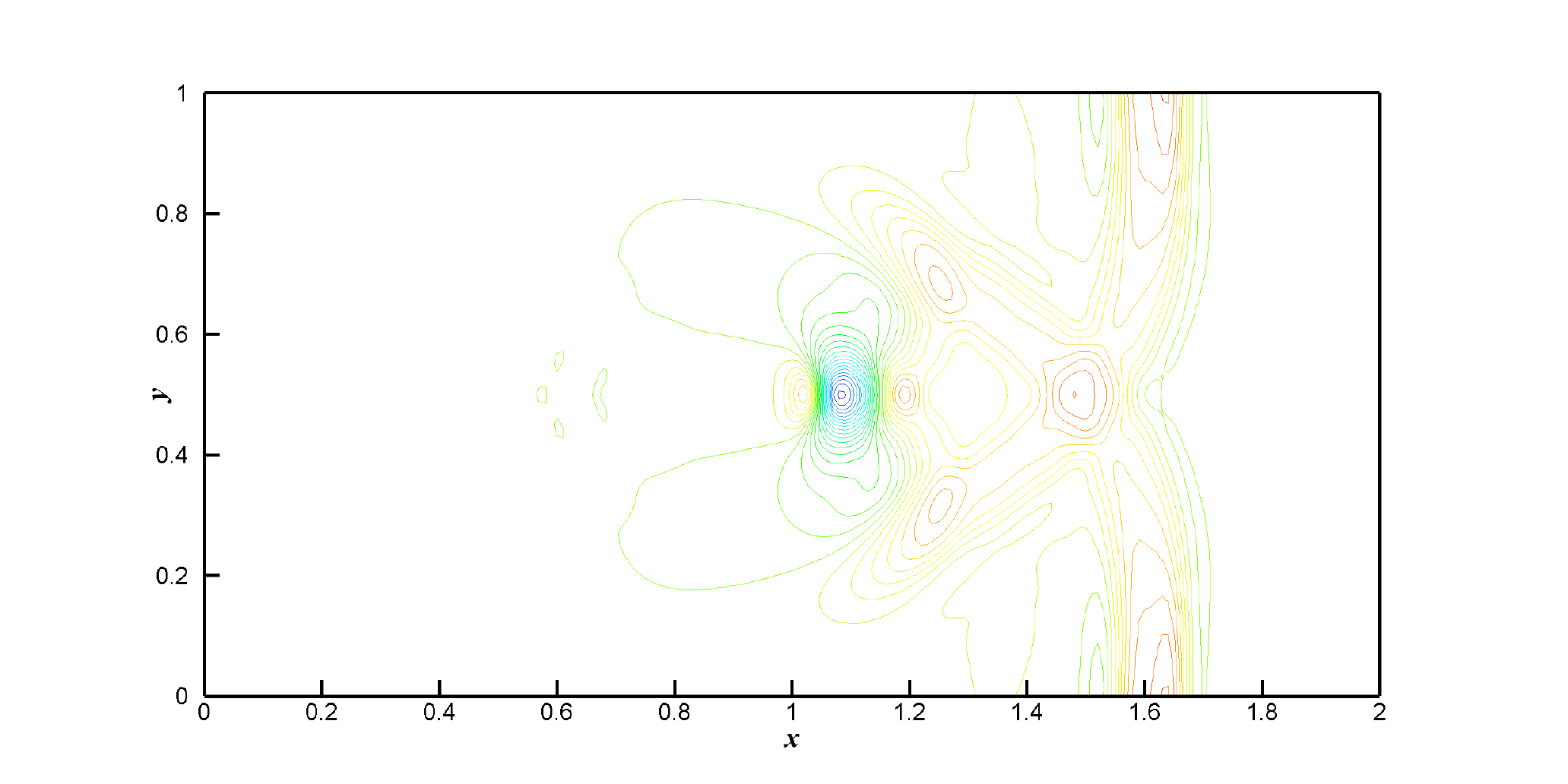}}
\subfigure[$N_x\times N_y=600 \times 300$, $t=0.48$]{
\includegraphics[width=0.45\textwidth,trim=20 10 30 30,clip]{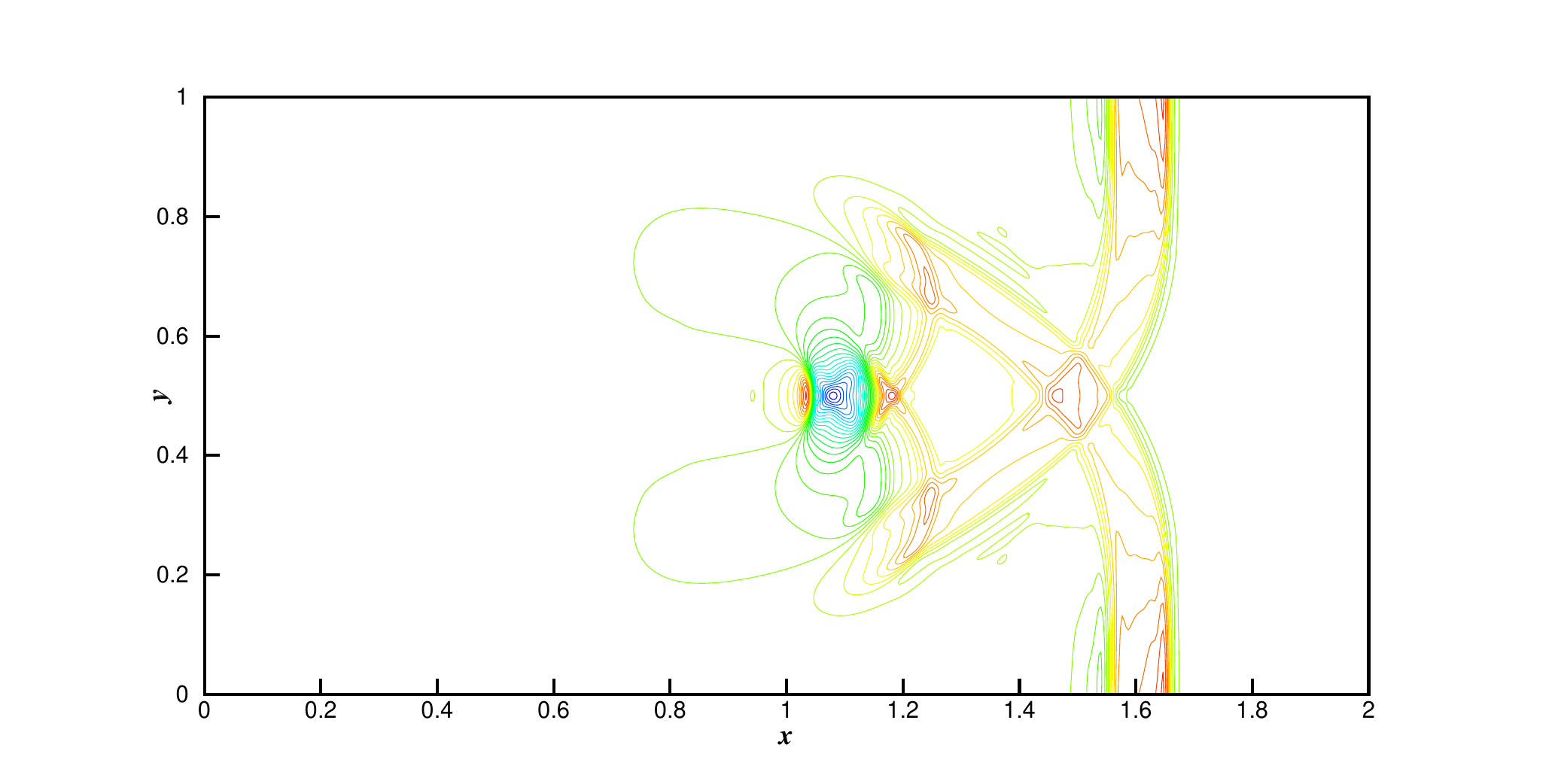}}
\subfigure[$N_x\times N_y=200 \times 100$, $t=0.60$]{
\includegraphics[width=0.45\textwidth,trim=20 10 30 30,clip]{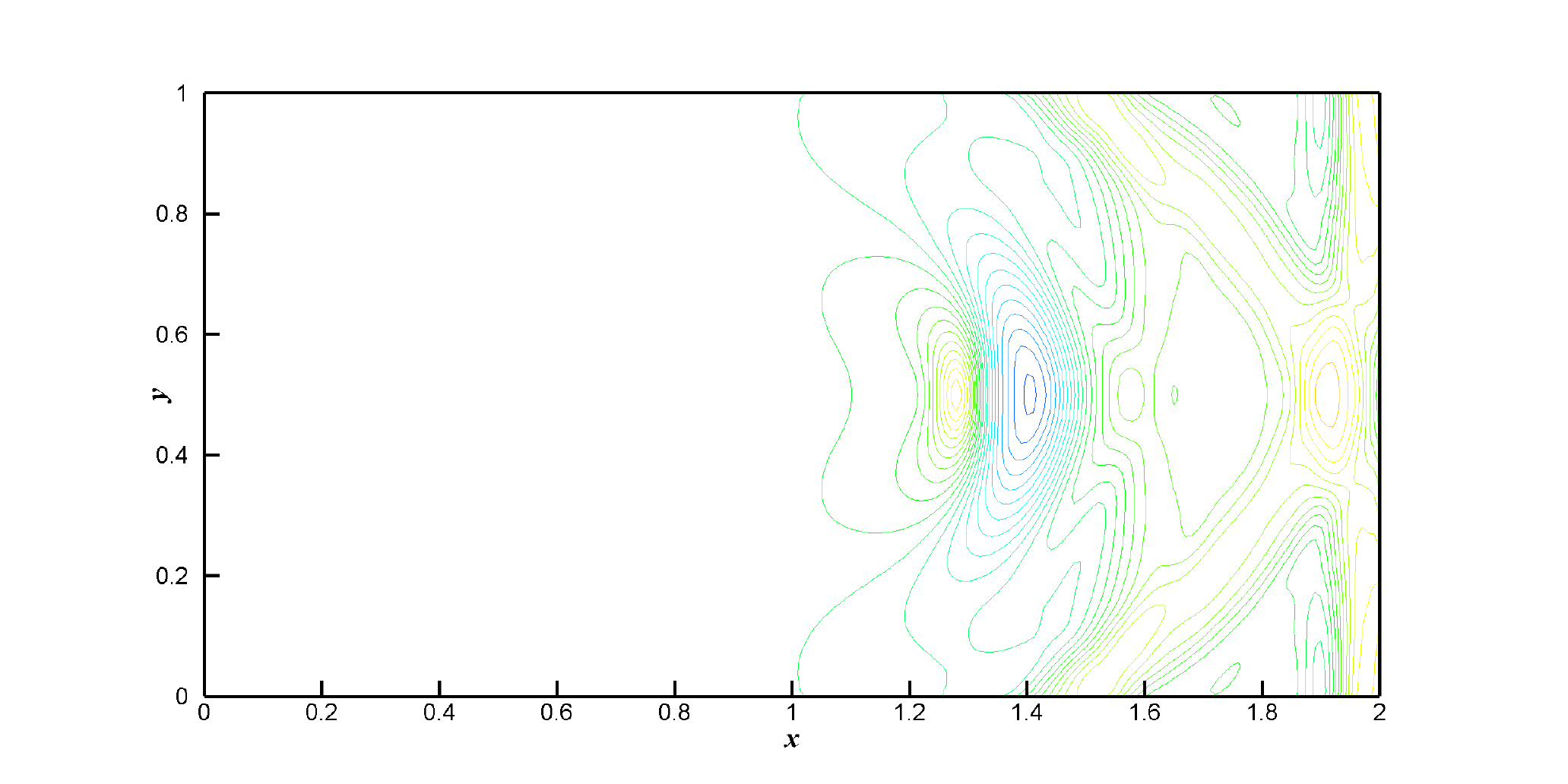}}
\subfigure[$N_x\times N_y=600 \times 300$, $t=0.60$]{
\includegraphics[width=0.45\textwidth,trim=20 10 30 30,clip]{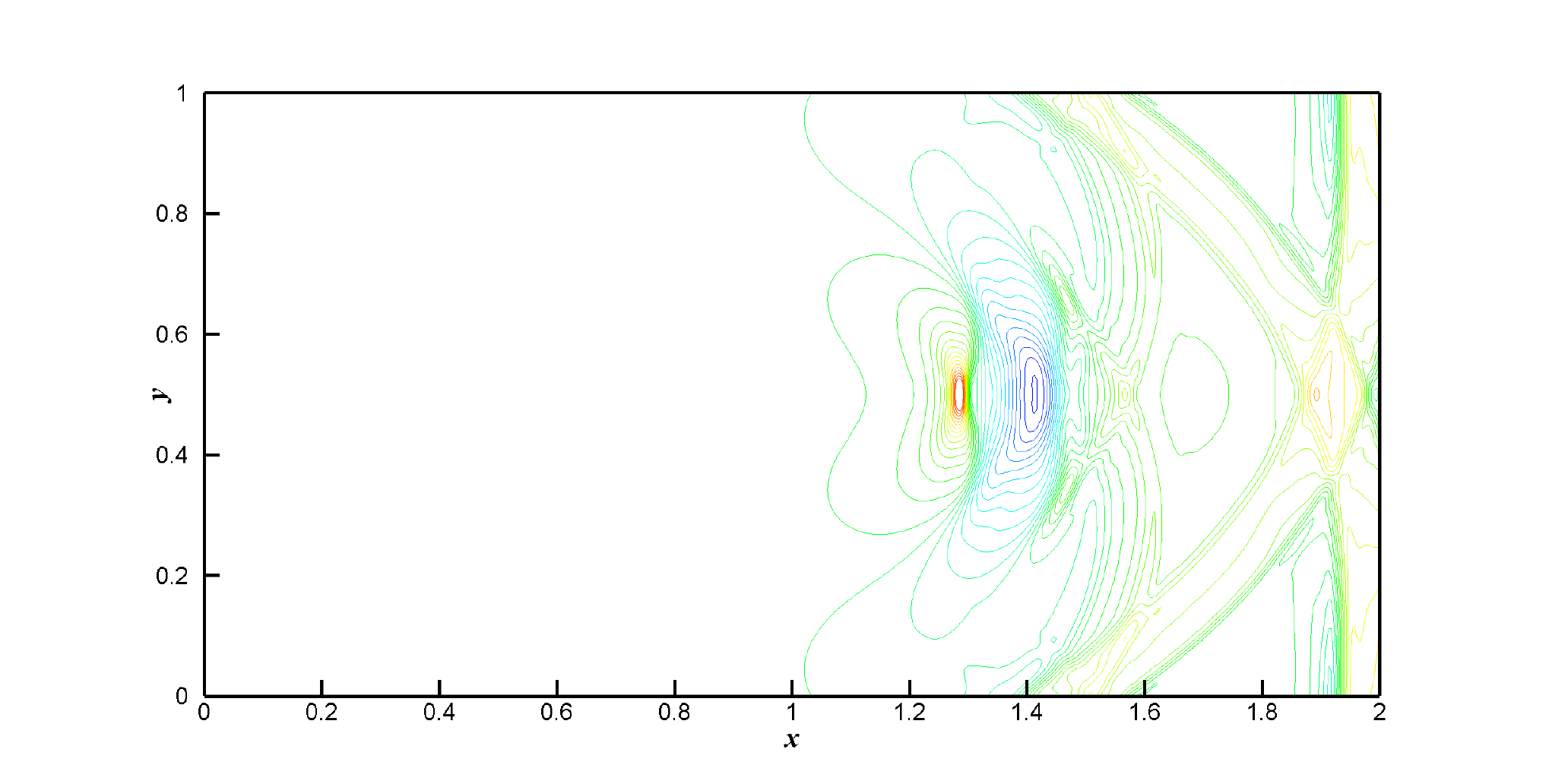}}
 \caption{Example \ref{test4-2d}. The contours of the free water surface level $h+b$ at different time obtained by the proposed HWENO scheme. $30$ uniformly spaced contour lines and the same ranges as the WENO-XS scheme \cite{Xing-Shu-2006JCP}.}
\label{Fig:test4-2d}
\end{figure}

\begin{example}\label{test5-2d}
(The circular Dam bread problem for the 2D SWEs.)
\end{example}
In this example we simulate the circular Dam bread problem \cite{Capilla-Balaguer-2013,Wang-etal-2020} for the 2D SWEs. The computational domain is $[0, 2]\times[0,2]$ and the bottom topography is given by
\begin{equation*}
b(x,y)=\begin{cases}
\frac{1}{8}(\cos(2\pi(x-0.5))+1)(\cos(2\pi y)+1), \ & \sqrt{(x-1.5)^2+(y-1)^2}\leq 0.5,\\
0,\ & \hbox{otherwise.}
\end{cases}
\end{equation*}
The initial conditions are
\begin{equation*}
\begin{split}
&h(x,y,0)=\begin{cases}
1.1-b(x,y), \quad & \sqrt{(x-1.25)^2+(y-1)^2}\leq 0.1,\\
0.6-b(x,y),\quad & \hbox{otherwise,}
\end{cases}
\\&u(x,y,0)=0,\quad \quad v(x,y,0)=0.
\end{split}
\end{equation*}

The final time is $t = 0.15$. The free water surface level $h+b$ at $t=0.15$ obtained by HWENO scheme with a mesh $200\times 200$ and the cut of the corresponding result along the line $y=1$ are plotted in Fig.~\ref{Fig:test5-2d}. We can clearly see that the HWENO scheme works well, giving well resolved and non-oscillatory solution.

\begin{figure}[H]
\centering
\subfigure[free water surface]{
\includegraphics[width=0.45\textwidth,trim=0 0 20 10,clip]{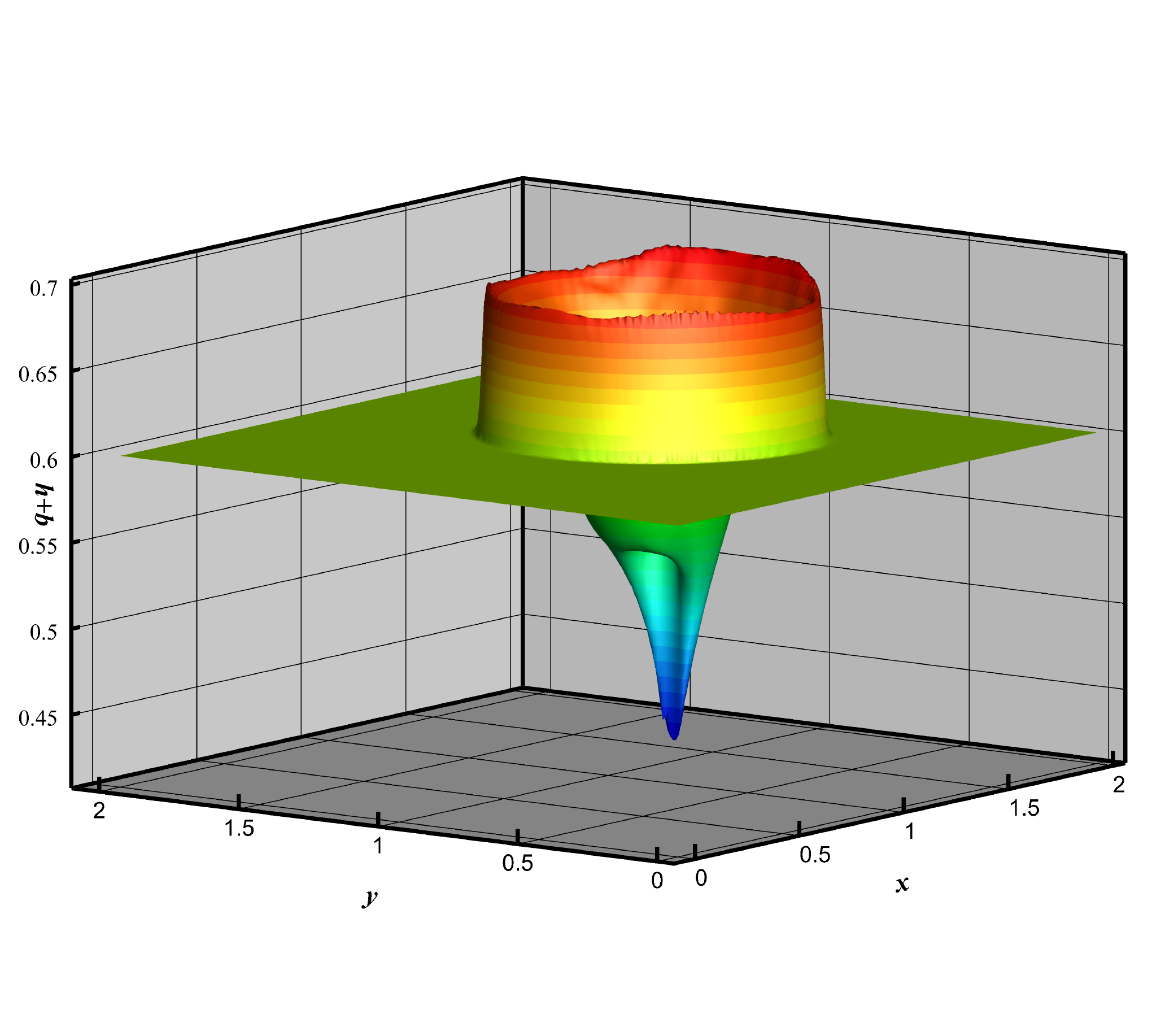}}
\subfigure[cut along $y=1$]{
\includegraphics[width=0.45\textwidth,trim=0 0 20 10,clip]{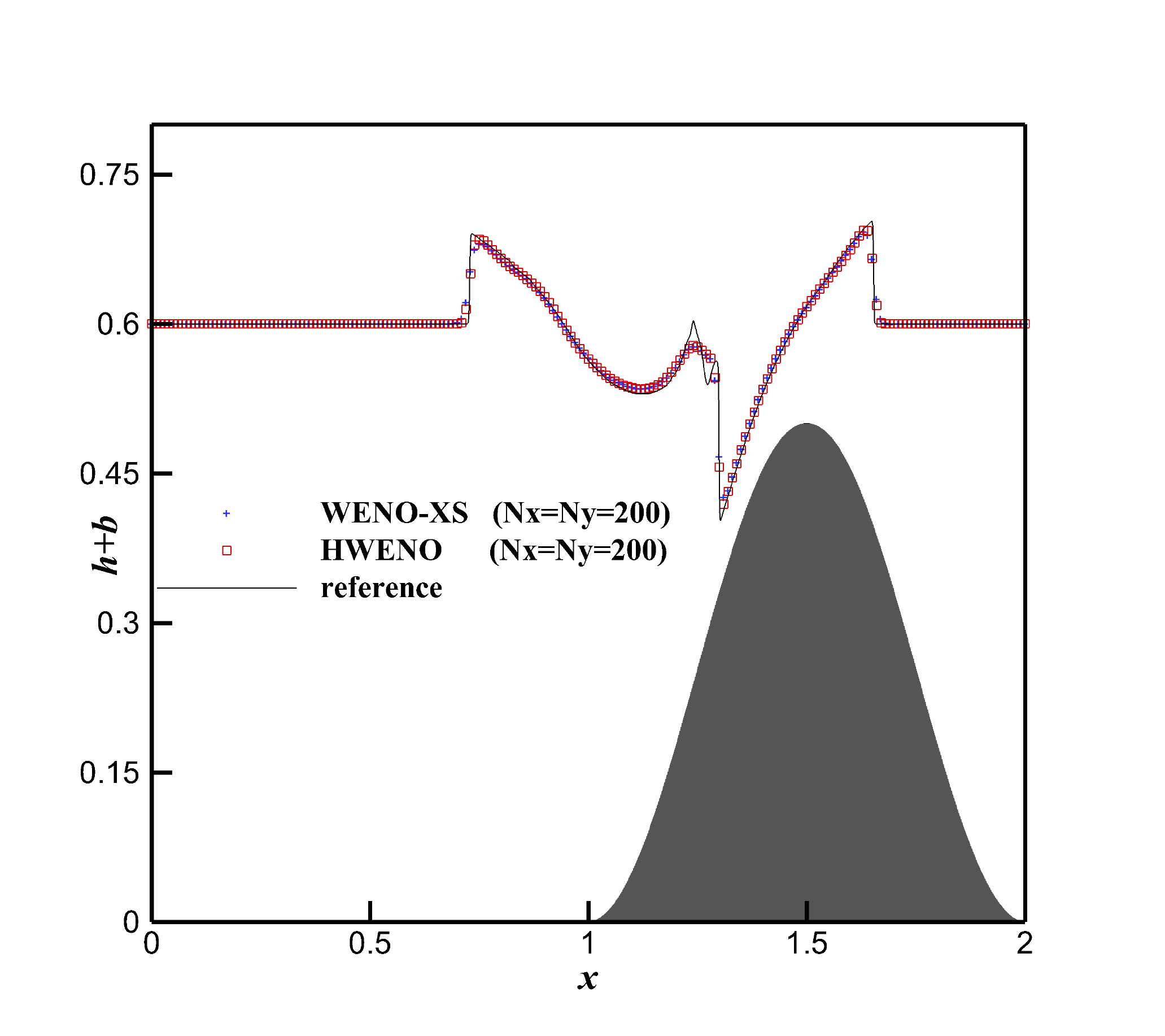}}
 \caption{Example \ref{test5-2d}. The free water surface level $h+b$ at $t=0.15$ obtained by the proposed HWENO scheme with $N_x\times N_y = 200\times 200$ and the cut of the corresponding results along the line $y=1$. }
\label{Fig:test5-2d}
\end{figure}

\section{Conclusions}
\label{sec:conclusions}

In this paper, we constructed a well-balanced fifth-order finite difference Hermite WENO (HWENO) scheme to solve the one- and two-dimensional shallow water equations with non-flat bottom topography, where both the function value and the derivative of
the equilibrium variable are evolved in time to make the scheme more compact.
Here, the similar idea of the WENO-XS scheme \cite{Xing-Shu-2005JCP} is used to achieve the well-balance property by balancing the flux gradients and the source terms.
Meanwhile, to control spurious oscillations, an HWENO limiter is applied for the derivatives of equilibrium variables in the time discretization step, which does not affect the well-balance property firstly.
And the HWENO limiter is vital for the stability of the proposed HWENO scheme, where lacking this limiter would lead to obvious oscillations near discontinuities, which has been shown in \cite{ZhaoZhuang-2020JSC-FD} for hyperbolic conservation laws. In addition, the proposed fifth-order HWENO scheme only needs a compact three-point stencil while the same order WENO-XS scheme \cite{Xing-Shu-2005JCP} needs a five-point stencil in the reconstruction.

Various benchmark examples in one and two dimensions are given to demonstrate the HWENO scheme has the properties of well-balance, fifth-order accuracy, non-oscillation, and high resolution.
The numerical results also show that the HWENO scheme is more accurate and efficient than the WENO-XS scheme.


\begin{thebibliography}{10}

\bibitem{Audusse-etal-2004Siam}
E.~Audusse, F.~Bouchut, M.-O. Bristeau, R.~Klein, and B.~Perthame,
A fast and stable well-balanced scheme with hydrostatic reconstruction for shallow water flows,
{\em SIAM J. Sci. Comput.}, 25 (2004), 2050-2065.

\bibitem{Bermudez-Vazquez-1994}
A.~Bermudez and M.~E. Vazquez,
Upwind methods for hyperbolic conservation laws with source terms,
{\em Comput.} \& {\em Fluid.}, 23 (1994), 1049-1071.

\bibitem{Caleffi-2011}
V. Caleffi,
A new well-balanced Hermite weighted essentially non-oscillatory scheme for shallow water equations,
{\em Int. J. Numer. Methods Fluids}, 67 (2011), 1135-1159.

 \bibitem{Capdeville}
G. Capdeville,
A Hermite upwind WENO scheme for solving hyperbolic conservation laws,
{\em J. Comput. Phys.}, 227 (2008), 2430-2454.

\bibitem{Capilla-Balaguer-2013}
M. T. Capilla and A. Balaguer-Beser,
A new well-balanced non-oscillatory central scheme for the shallow water equations on rectangular meshes,
 {\em J. Comput. Appl. Math.}, 252 (2013), 62-74.



\bibitem{ChengLCS-2022}
M. Cheng, L. Tang, Y. Chen, and S. Song,
A well-balanced weighted compact nonlinear scheme for
shallow water equations on curvilinear grids,
{\em J. Comput. Phys.}, 463 (2022), 111250.

\bibitem{GaoH-2017-FD}
Z. Gao and  G. Hu,
High order well-balanced weighted compact nonlinear schemes for shallow water equations, {\em Commun. Comput. Phys.}, 22 (2017), 1049-1068.


\bibitem{HuangXX-2022}
G. Huang, Y. Xing, and T. Xiong,
High order well-balanced asymptotic preserving finite difference WENO schemes for the shallow water equations in all Froude numbers,
{\em J. Comput. Phys.}, 463 (2022), 111255.
\bibitem{js} G.-S. Jiang and C.-W. Shu,
Efficient implementation of weighted ENO schemes,
{\em J. Comput. Phys.}, 126 (1996), 202-228.

\bibitem{LeVeque-1998JCP}
R.~LeVeque,
Balancing source terms and flux gradients in high-resolution Godunov methods: the quasi-steady wave-propagation algorithm,
{\em J. Comput. Phys.}, 146 (1998), 346-365.


\bibitem{Li-Lu-Qiu-2012JSC}
G. Li, C. Lu, and J. Qiu,
Hybrid well-balanced WENO schemes with different indicators for shallow water equations,
{\em J. Sci. Comput.}, 51 (2012),  527-559.

\bibitem{Li-etal-2018JCAM}
G.~Li, L.~Song, and J.~Gao,
High order well-balanced discontinuous Galerkin methods based on hydrostatic reconstruction for shallow water equations,
{\em J. Comput. App. Math.}, 340 (2018), 546--560.

\bibitem{LiMRHW1}
J. Li, C.-W. Shu, and J. Qiu,
Multi-resolution HWENO schemes for hyperbolic conservation laws,
{\em J. Comput. Phys.}, 446 (2021), 110653.

\bibitem{Li-DonGao-2020}
P. Li, W. Don, and Z. Gao,
High order well-balanced finite difference WENO interpolation-based schemes for shallow water equations,
{\em Comput.} \& {\em Fluid.}, 201 (2020), 104476.

\bibitem{LiuQ1}
H. Liu and J. Qiu,
Finite difference Hermite WENO schemes for conservation laws,  {\em J. Sci. Comput.}, 63 (2015), 548-572.

\bibitem{LuQiu-2011JSC}
C. Lu and J. Qiu,
Simulations of shallow water equations with finite difference Lax-Wendroff weighted essentially non-oscillatory schemes,
{\em J. Sci. Comput.}, 47 (2011), 281-302.
\bibitem{MW}
Z. Ma and S. P. Wu,
HWENO schemes based on compact difference for hyperbolic conservation laws,
{\em J. Sci. Comput.}, 76 (2018), 1301-1325.


\bibitem{NoelleXS-2007}
S. Noelle, Y. Xing, and C.-W. Shu,
High-order well-balanced finite volume WENO schemes for shallow water equation with moving water,
{\em J. Comput. Phys.}, 226 (2007), 29-58.


\bibitem{QS-HWENO-2004}
J. Qiu and C.-W.Shu,
Hermite WENO schemes and their application as limiters for Runge-Kutta discontinuous Galerkin method: one-dimensional case,
{\em J. Comput. Phys.}, 193 (2004), 115-135.

\bibitem{QS-HWENO-2005}
J. Qiu and C.-W. Shu,
Hermite WENO schemes and their application as limiters for Runge-Kutta discontinuous Galerkin method II: two dimensional case,
{\em Comput.} \& {\em Fluid.}, 34 (2005), 642-663.

\bibitem{TLQ}Z. Tao, F. Li, and J. Qiu,
High-order central Hermite WENO schemes: dimension-by-dimension moment-based reconstructions,
    {\em J. Comput. Phys.}, 318 (2016), 222-251.

\bibitem{HzTang-2004}
H. Tang,
Solution of the shallow-water equations using an adaptive moving mesh method,
 {\em Int. J. Numer. Meth. Fluids}, 44 (2004), 789-810.

\bibitem{TangTX-2004}
H. Tang, T. Tang, and K. Xu,
A gas-kinetic scheme for shallow-water equations with source terms,
{\em Z. angew. Math. Phys.}, 55 (2004), 365-382.


\bibitem{Vukovic-Sopta-2002JCP}
S. Vukovic and L. Sopta,
ENO and WENO schemes with the exact conservation property for one-dimensional shallow water equations,
{\em J. Comput. Phys.}, 179 (2002), 593-621.

\bibitem{Wang-etal-2020}
Z. Wang, J. Zhu, and N. Zhao,
A new fifth-order finite difference well-balanced multi-resolution WENO scheme for solving shallow water equations,
{\em Comput. Math. Appl.}, 80 (2020), 1387-1404.

\bibitem{weHTENO}
I. Wibisono and A. K. Engkos,
Fifth-order Hermite targeted essentially non-oscillatory schemes for hyperbolic conservation laws, {\em J. Sci. Comput.}, 87 (2021), 1-23.

\bibitem{Xing-Shu-2005JCP}
Y.~Xing and C.-W. Shu,
High order finite difference WENO schemes with the exact conservation property for the shallow water equations,
{\em J. Comput. Phys.}, 208 (2005), 206-227.

\bibitem{Xing-Shu-2006JCP}
Y.~Xing and C.-W. Shu,
High order well-balanced finite volume WENO schemes and discontinuous Galerkin methods for a class of hyperbolic systems with source terms,
{\em J. Comput. Phys.}, 214 (2006), 567-598.

\bibitem{Xing-Shu-2006CiCP}
Y.~Xing and C.-W. Shu,
A new approach of high order well-balanced finite volume WENO schemes and discontinuous Galerkin methods for a class of hyperbolic systems with source terms,
{\em Comm. Comput. Phys.}, 1 (2006), 100-134.

\bibitem{Xing-Zhang-Shu-2010}
Y.~Xing, X.~Zhang, and C.-W. Shu,
Positivity-preserving high order well-balanced discontinuous Galerkin methods for the shallow water equations,
{\em Adv. Water Resourc.}, 33 (2010), 1476-1493.

\bibitem{Xing-Zhang-2013JSC}
Y.~Xing and X.~Zhang.
Positivity-preserving well-balanced discontinuous Galerkin methods for the shallow water equations on unstructured triangular meshes,
{\em J. Sci. Comput.}, 57 (2013), 19-41.


\bibitem{ZA}
 Y. H. Zahran and A. H. Abdalla,
 Seventh order Hermite WENO scheme for hyperbolic conservation laws,
 {\em Comput.} \& {\em Fluid.}, 131 (2016), 66-80.

\bibitem{Zhang-Huang-Qiu-2021JSC}
M. Zhang, W. Huang, and J. Qiu,
A high-order well-balanced positivity-preserving moving mesh DG method for the shallow water equations with non-flat bottom topography,
{\em J. Sci. Comput.}, 87 (2021), 88.

\bibitem{Zhang-Huang-Qiu-2020CiCP}
M. Zhang, W. Huang, and J. Qiu,
A well-balanced positivity-preserving quasi-Lagrange moving mesh DG method for the shallow water equations,
{\em Commun. Comput. Phys.}, 31 (2022), 94-130.


\bibitem{Zhang-Xia-Xu-2021JSC}
W. Zhang, Y. Xia, and Y. Xu,
Positivity-preserving well-balanced arbitrary Lagrangian-Eulerian discontinuous Galerkin methods for the shallow water equations,
{\em J. Sci. Comput.}, 88 (2021), 57.

\bibitem{ZhaoZhuang-2020JSC-FD}
Z. Zhao, Y.-T. Zhang, and J. Qiu,
A modified fifth order finite difference Hermite WENO scheme for hyperbolic conservation laws,
{\em J. Sci. Comput.}, 85 (2020), 29.

\bibitem{Zhou-etal-2001JCP}
J.~G. Zhou, D.~M. Causon, C.~G. Mingham, and D.~M. Ingram,
\newblock The surface gradient method for the treatment of source terms in the
  shallow-water equations,
\newblock {\em J. Comput. Phys.}, 168 (2001), 1-25.
\end{thebibliography}
\end{document}